\documentclass[12pt]{article}

\usepackage{amsfonts, amscd,amsmath, amssymb, fancyhdr, graphicx}
\usepackage[title]{appendix}
\usepackage{color}
\usepackage{bbm}
\usepackage{verbatim}
\usepackage{enumitem}
\usepackage{graphicx,caption}

\usepackage[margin=1in,footskip=0.5in]{geometry}
\numberwithin{equation}{section}
\usepackage[cmtip,all]{xy}
\usepackage[pdftex]{hyperref}

\renewcommand{\leftmark}{{\scriptsize Nilpotent Higgs bundles and families of flat connections}}

\def\eqref#1{(\ref{#1})}

\newcommand{\Z}{{\mathbb Z}}
\def\C{{\mathbb C}}
\def\O{{\mathcal O}}
\def\CP{{\mathbb{CP}}}
\newcommand{\R}{{\mathbb{R}}}

\def\1{\sqrt{-1}\:}

\newcommand{\SL}{{\mathrm{SL}}}

\newcommand{\GL}{{\mathrm{GL}}}

\newcommand{\MdR}{\mathcal{M}_\text{dR}}

\newcommand{\MHod}{\mathcal{M}_\text{Hod}}
\newcommand{\MH}{\mathcal{M}_\text{H}}
\newcommand{\MDH}{\mathcal{M}_\text{DH}}

\renewcommand{\bar}{\overline}
\renewcommand{\phi}{\varphi}
\renewcommand{\epsilon}{\varepsilon}
\renewcommand{\geq}{\geqslant}
\renewcommand{\leq}{\leqslant}
\renewcommand{\max}{{\rm max}}

\newcommand{\End}{\operatorname{End}}

\newcommand{\Hom}{\operatorname{Hom}}

\newcommand{\Hol}{\operatorname{Hol}}

\newcommand{\diag}{\operatorname{\sf diag}}

\newcommand{\rk}{\operatorname{rk}}

\newcommand{\Tr}{\operatorname{Tr}}

\newcommand{\CL}{\mathcal{CL}}
\newcommand{\delbar}{\bar{\partial}}

\renewcommand{\Re}{\operatorname{Re}}
\renewcommand{\Im}{\operatorname{Im}}

\usepackage{mathtools}
\usepackage{amsthm}

\newtheorem{theorem}{Theorem}[section]

\newtheorem{cor}[theorem]{Corollary}
\newtheorem{prop}[theorem]{Proposition} 
\newtheorem{lem}[theorem]{Lemma} 
\newtheorem{thm}[theorem]{Theorem} 

\theoremstyle{definition}
\newtheorem{definition}{Definition}[section]

\newtheorem{rmk}[definition]{Remark}

\newtheorem{deff}[definition]{Definition}

\makeatletter

\@addtoreset{equation}{section} \@addtoreset{footnote}{section}
\makeatother

\begin{document}
%%%%%%%%%%%%%%%%%%%%%%%%%%%%%%%%%%%%%%%%%%%%%%%%%%%%%%%%%%%%
\begin{center}
{\LARGE\bf
Nilpotent Higgs bundles and families of flat connections\\[3mm]
}
%%%%%%%%%%%%%%%%%%%%%%%%%%%%%%%%%%%%%%%%%%%%%%%%%%%%%%%%%%%%

   Sebastian Schulz %\footnote{The work of S.~Schulz is partially supported by NSF grants DMS 1107452, 1107263, 1107367 "RNMS: GEometric structures And Representation varieties" (the GEAR Network)."

 %{\bf Keywords:} Higgs bundles,  orthogonal representations, spectral data  Hitchin fibration, $(B,A,A)$-branes.

%{\bf 2000 Mathematics Subject Classification:} }

\end{center}

%%%%%%%%%%%%%%%%%%%%%%%%%%%%%%%%%%%%%%%%%%%%%%%%
{\small \hspace{0.08\linewidth}
\begin{minipage}[t]{0.84\linewidth}
{\bf Abstract} \\  We investigate $\C^\times$-families of flat connections whose leading term is a nilpotent Higgs field. Examples of such families include real twistor lines and families arising from the conformal limit. We show that these families have the same monodromy as families whose leading term is a regular Higgs bundle and use this to deduce that traces of holonomies are asymptotically exponential in rational powers of the parameter of the family.  \end{minipage}
}
%%%%%%%%%%%%%%%%%%%%%%%%%%%%%%%%%%%%%%%%%%%%%%%%

\tableofcontents

%%%%%%%%%%%%%%%%%%%%%%%%%%%%%%%%%%%%%%%%%%%%%%%%
%\newpage
\section{Introduction} \label{sec:intro}

\subsection{Higgs bundles and the WKB method}
The theory of Higgs bundles is intimately intertwined with the theory of families of flat connections. Given a Higgs bundle $(\mathcal{E}, \phi)$ on a Riemann surface $C$, there are several $\C^\times$-families $$\nabla_\epsilon = \epsilon^{-1} \phi + D + \epsilon \psi$$ of flat connections naturally associated to it, for example through the \textit{Non-Abelian Hodge Correspondence} and through the \textit{Conformal Limit} construction. It is natural to ask how these the monodromy of these connections behave in the $\epsilon \to 0$ limit. The \textit{exact WKB method} \cite{Voros83} is an important tool in tackling problems of this sort and has played a pivotal role  in shedding light on a variety of geometric structures over the last decade.

Firstly, Gaiotto-Moore-Neitzke~\cite{GMN09} proposed a conjectural description of the asymptotics of the Hitchin system's Hyperk\"ahler metric that was motivated by the WKB method. Their conjecture has since been confirmed in many different settings \cite{MSSW16, MSSW, DumasNeitzke, Fredrickson18A, FMSW}. 

Secondly, Iwaki-Nakanishi \cite{IwakiNakanishi, IwakiNakanishi2} have related the exact WKB method to the theory of \textit{Cluster algebras}. Cluster algebras, first introduced by Fomin-Zelevinsky~\cite{FZCluster}, have proven ubiquitous in Higher Teichm\"uller Theory thanks to the work of Fock-Goncharov~\cite{FockGoncharov} and are linked to a variety of facets of mathematics, see e.g.~\cite{KellerClusterIntro} for a survey.

Thirdly, though not unrelated, WKB analysis is well-known to be related to the study of Donaldson-Thomas invariants and their wall-crossing structure \cite{AllegrettiVorosSymbols, AllegrettiBridgeland, Bridgeland_RH_from_DT, BridgelandSmith}.

All of these applications are subject to a certain genericity condition, namely that the Higgs field can be diagonalized with distinct eigenvalues, at least away from a finite set of points of $C$. Under this condition, the exact WKB method gives the following: For any closed loop $\gamma$ on $C$, the trace of the holonomy around $\gamma$ grows exponentially in $\epsilon^{-1}$ in the $\epsilon \searrow 0$ limit\footnote{That is $\epsilon \to 0$ while $\epsilon > 0$.}. More precisely, there exists a constant $Z_\gamma \in \C^\times$, $\Re (Z_\gamma) >0$ such that
\begin{equation} \label{eq:HolWKBcurve}
    \lim_{\epsilon \searrow 0} \big( \Tr \Hol_\gamma (\nabla_\epsilon) \cdot \exp (-\epsilon^{-1} Z_\gamma )  \big)
\end{equation}
exists and is non-zero, see Section~\ref{subsec:WKBreview} for a more detailed discussion. The constant $Z_\gamma$ admits a geometric interpretation as a period of the \textit{spectral curve}, a certain (ramified) cover of $C$ determined by the Higgs bundle $(\mathcal{E}, \phi)$.

%in at least two important ways, namely as a  \textit{Real Twistor Line} or through the \textit{Conformal Limit}. In either case, $D$ is a (typically non-flat) connection on $\mathcal{E}$ whose $(0,1)$-part $D^{(0,1)}$ coincides with the holomorphic structure $\delbar_E$ on $\mathcal{E}$, and $\psi$ is a certain $\End (\mathcal{E})$-valued 1-form. It is natural to ask how these connections degenerate in the $\epsilon \to 0$ limit.

%A Higgs field $\phi$ is a certain vector bundle endomorphism that can generically be diagonalized with distinct eigenvalues. Under this condition, the small $\epsilon$ limit can be studied by the \textit{exact WKB method} (\textcolor{red}{References}). One implication of this analysis is the following: 

A similar description without this genericity assumption has so far been unknown, and in this paper we deal with the most degenerate case of a \textit{nilpotent Higgs bundle}, meaning one for which the eigenvalues of $\phi$ are all identically zero. The exact WKB method is not directly applicable in this situation because it is unclear what the analogue of the spectral curve should be. Moreover, the constant $Z_\gamma$ in $\ref{eq:HolWKBcurve}$ is an integral of the eigenvalue of $\phi$ with the largest real part, which in this case would simply yield zero.

\subsection{Families of flat connections}

Before we proceed, let us briefly explain what kind of families of flat connections we consider. A Higgs bundle on a Riemann surface $C$ is a pair $(\delbar_E, \phi)$ consisting of a holomorphic vector bundle $(E, \delbar_E)$ and a holomorphic endomorphism $\phi \in H^0 (C, End(E) \otimes K_C)$ where $K_C$ denotes the holomorphic cotangent bundle of $C$. A Higgs bundle generically obeys a stability condition analogous to that for vector bundles, and for stable Higgs bundles there exists a unique hermitian metric $h$ such that 
\begin{equation} \label{eq:IntroTwistorLine}
    \nabla_\zeta = \zeta^{-1} \phi + D_h + \zeta \phi^{\dagger_h}
\end{equation}
is a flat connection for all $\zeta \in \C^\times$, where $D_h$ denotes the Chern connection of the pair $(\delbar_E, h)$. This is a consequence of the Non-Abelian Hodge Correspondence and $\nabla_\zeta$ is known as a \textit{real twistor line}; see Section~\ref{sec:NAHT} for more details.

We shall also be interested in a family of flat connections known as the \textit{Conformal Limit} family, as follows: For any $R \in \mathbb{R}_{>0}$ there exists a hermitian metric $h_R$ such that
\begin{equation}
    \nabla (\zeta, R) := \zeta^{-1} R \phi + D_{h_R} + \zeta R \phi^{\dagger_{h_R}}
\end{equation}
is a flat connection for all $\zeta \in \C^\times$. Letting $\hbar = R^{-1} \zeta$ this becomes a family
\begin{equation}
    \nabla_{\hbar, R} = \hbar^{-1} \phi + D_{h_R} + \hbar R^2 \phi^{\dagger_{h_R}}.
\end{equation}
The $R \to 0$ limit of this family is known as the $\hbar$-\textit{conformal limit} $\CL_\hbar(\delbar_E, \phi)$ of the Higgs bundle $(\delbar_E, \phi)$. We provide more details in Section~\ref{sec:CL}.

\subsection{Main results}

Our main theorems concern a weaker version of WKB theory that is not applicable to all curves $\gamma$ but only to curves we call \textit{WKB curves}, see section~\ref{subsec:WKBreview}. It is important to note that WKB curves exist on any Riemann surface $C$ and for any non-nilpotent $\SL (2, \C)$-Higgs bundle $(\delbar_E, \phi)$ on it, see Appendix~\ref{app:WKBcurves}. Also, recall that the moduli space of Higgs bundles $\MH$ admits a $\C^\times$-action given by $\xi \cdot (\delbar_E, \phi) = (\delbar_E, \xi \cdot \phi)$. Then we prove that the holonomy of a real twistor line around $\gamma$ grows exponentially in $\zeta^{-1/2}$. More precisely, 

\medskip

\noindent \textbf{Theorem \ref{thm:WKBInRank2}} \textit{For a nilpotent $\SL(2, \C)$-Higgs bundle that is not a fixed point of the $\C^\times$-action, there exists a closed curve $\gamma$ on $C$ and a $Z_\gamma \in \C$, $\Re (Z_\gamma ) >0$ such that the limit
\begin{equation} \label{intro:MainLimit}
    \lim_{\zeta \searrow 0} \big( \Tr \Hol_\gamma (\nabla_\zeta) \cdot \exp (-\zeta^{-1/2} Z_\gamma )  \big),
\end{equation}
exists and is non-zero. Here, $\Hol_\gamma$ denotes the holonomy around $\gamma$ and $\nabla_\zeta$ is the family of flat connections (\ref{eq:IntroTwistorLine}).} 

\medskip

In fact, we will show that $Z_\gamma$ is a period of the spectral curve $\Sigma'$ of a non-nilpotent Higgs bundle $(\delbar'_E, \Phi)$ that is canonically obtained from $(\delbar_E, \phi)$. Any curve $\gamma$ that is a WKB curve with respect to $(\delbar'_E, \Phi)$ will obey~(\ref{intro:MainLimit}).

Theorem~\ref{thm:WKBInRank2} holds in slightly more generality then we have stated here. For example, we show that the analogous statement is true for the conformal limit family $\mathcal{CL}_\hbar (\delbar_E, \phi)$ if the vector bundle $(E, \delbar_E)$ is itself stable.

Let us explain how this \textit{secondary Higgs bundle} $(\delbar'_E, \Phi)$ arises. If $\phi \neq 0$ then $(\delbar_E, \phi)$ defines a $C^\infty$-decomposition $E = K \oplus L$ into line bundles, where $K = \ker \phi$ and $L = K^{\bot_h}$ is the orthogonal complement with respect to the harmonic metric $h$. The matrix $g(\epsilon) = \diag (\epsilon^{-1/4}, \epsilon^{1/4})$ defines a gauge transformation via conjugation, and applied to the family $\nabla_\epsilon$ it yields a new family
\begin{equation} \label{eq:IntroNewFamily}
    \nabla'_\epsilon := g(\epsilon)^{-1} \cdot \nabla_\epsilon \cdot g(\epsilon) = \epsilon^{-1/2} \Phi + \delbar'_E + \partial'_E + O (\epsilon^{1/2}).
\end{equation}
We use this construction and prove the following 

\medskip

\noindent \textbf{Proposition \ref{prop:SL2TwistorLine}} \textit{Let $\nabla_\zeta$ be the real twistor line of a stable, nilpotent Higgs bundle $(\delbar_E, \phi)$. Then there exists a gauge-equivalent family \begin{equation}
 \nabla'_\zeta = \zeta^{-1/2} \Phi + \delbar'_E + \partial'_E + O (\zeta^{1/2}) 
\end{equation}
of flat connections such that $(\delbar'_E, \Phi )$ defines a holomorphic, stable Higgs bundle. Moreover, $\Phi$ is nilpotent if and only if $(\delbar_E, \phi)$ is a fixed point of the $\C^\times$-action.} 

\medskip

We prove the analogous statement for a conformal limit family $\mathcal{CL}_\hbar (\delbar_E, \phi)$ provided that the vector bundle $(E, \delbar_E)$ is itself stable (Proposition \ref{prop:SL2CL} below). It is instructive to note that the gauge transformation~(\ref{eq:IntroNewFamily}) does not necessarily preserve the \textit{type} of the family. For example, if $\nabla_\zeta$ represents a real twistor line then $\nabla'_\zeta$ does not because it is not real (Proposition~\ref{prop:NotATwistorLine} below). If $\nabla_\hbar$ arises from the conformal limit, then $\nabla'_\hbar$ has terms of powers $1/2$, $1$ and $3/2$ in $\hbar$.

Proposition \ref{prop:SL2TwistorLine} is the key step to proving Theorem \ref{thm:WKBInRank2}, as it allows to use the WKB method on the family $\nabla'_\zeta$. It is worth noting that the WKB method is typically employed on families that contain only integer powers of $\zeta$. There is no major obstruction to using the WKB method on families that contain rational powers of $\zeta$ and this was recently described by Thomas \cite{ThomasRationalWKB} as the \textit{Rational WKB method} and used to construct a formal embedding of the cotangent bundle of the space of higher complex structures into the character variety. It would be very interesting to further investigate the connection between these two different appearances of the rational WKB method.

As an application of Theorem~\ref{thm:WKBInRank2} we prove part of a conjecture of Simpson~\cite{Simpson10} concerning a natural stratification of the moduli space of flat connections $\MdR$ known as the \textit{partial oper stratification}. Simpson conjectured that the individual leaves of the stratification are closed inside of $\MdR$ which we answer affirmatively for the leaves inside of the unique open stratum in Corollary~\ref{cor:SimpsonConj} in the $\SL(2, \C )$ case. Note that this was already known for a minimal dimensional stratum known as the \textit{space of opers}.

\subsection{Generalizations}

We subsequently discuss the generalization to higher rank Higgs bundles. A novelty is a certain integer $m$ ($2 \leq m \leq n$) that is an invariant of a nilpotent $\SL (n, \C)$-Higgs bundle which is not a fixed point of the $\C^\times$-action. We restrict our attention to the family of real twistor lines $\nabla_\zeta $ and prove 

\medskip

\noindent \textbf{Theorem~\ref{thm:HigherRank}} \textit{Let $(\delbar_E, \phi)$ be a stable nilpotent $\SL(n, \C)$-Higgs bundle that is not a $\C^\times$-fixed point. Then there exists an integer $m$, $2 \leq m \leq n$, such that the associated real twistor line $\nabla_\zeta$ from~(\ref{eq:IntroTwistorLine}) is gauge-equivalent to a family of the form
\begin{equation}
  \nabla'_\zeta =  \zeta^{(1-m)/m} \Phi + \dots + \delbar'_E + \partial'_E + \dots
\end{equation}
and the pair $(\delbar'_E, \Phi)$ defines a stable $\SL(n, \C)$-Higgs bundle that is not nilpotent.} 

\medskip

As a consequence, we expect that asymptotically
\begin{equation}
    \Tr \Hol_\gamma (\nabla_\zeta) = \Tr \Hol_\gamma (\nabla'_\zeta) \sim \exp (\zeta^{(1-m)/m} Z_\gamma)
\end{equation}
for some constant $Z_\gamma$. This is based on the higher rank WKB method which is more complicated and less well-developed.

Lastly, we discuss parabolic $\SL(2, \C)$-Higgs bundles and show that the statements of Proposition~\ref{prop:SL2TwistorLine} and Theorem~\ref{thm:WKBInRank2} hold for parabolic twistor lines. The main difficulty is that one has to make sure that the new family $\nabla'_\zeta$ behaves well near the puncture, i.e.\ that the constituents of $\nabla'_\zeta$ have the correct singular behavior.

Though we do not discuss it in detail here, the generalization to $\mathrm{GL}(n, \C)$-Higgs bundles should be immediate. It would be interesting to see the generalization to $G_\C$-Higgs bundles for arbitrary reductive groups $G_\C$ over $\C$.

\medskip

\noindent \textbf{Structure:} This paper is organized as follows: In Section~\ref{sec:Background} we recall necessary background about the Hitchin system. In Section~\ref{section:NilpotentWKB} we study the $\SL (2, \C)$-case described above, this section contains most of the novel ideas of this work. We generalize our results to the case of $\SL(n, \C)$ in Section~\ref{sec:SLN} and to parabolic $\SL (2, \C)$-Higgs bundles in section~\ref{sec:Parabolic}. We discuss the nilpotent cone of the moduli space of Higgs bundles on a four-punctured sphere, dubbed the \textit{toy model} for its simplicity, in Section~\ref{sec:ToyModel} and include in Appendix~\ref{app:WKBcurves} a review on (Half-) Translation surfaces that is key to our construction of WKB curves.

\medskip

\noindent \textbf{Acknowledgments:} I am most grateful to my advisor, Andy Neitzke, without whom none of this would have been possible. Not only did he suggest pursuing this problem, but he has been a constant source of encouragement and inspiration along the way, and I sincerely thank him for his guidance, insights and patience. I am indebted to Sam Raskin for his help, and would like to thank Sebastian Heller, Pengfei Huang and Mike Wolf for useful correspondences.

%\newpage
\section{The geometry of the Hitchin system} \label{sec:Background}
In this section we review the definition of Higgs bundles as well as the geometric properties of the moduli space of (stable) Higgs bundles. We will focus on the case $G = \SL (n, \C)$ and fix a closed, connected Riemann surface $C$ whose canonical bundle is denoted by $K_C$, as well as a complex vector bundle $E \to C$ of rank $n$ together with a trivialization of its determinant line bundle $\Lambda^n E$.
\subsection{The Hitchin system}
\begin{definition}
An $\SL (n, \C)$-Higgs bundle is a pair $(\bar{\partial}_E, \phi)$ consisting of
\begin{enumerate}[label=(\roman*)]
    \item a holomorphic structure $\bar{\partial}_E$ on $E$, together with a trivialization of its determinant line bundle $\det (E, \bar{\partial}_E ) \xrightarrow{\sim} \mathcal{O}_C$, and
    \item a traceless, holomorphic section $\phi \in H^0 (C, \End (E) \otimes K_C )$.
\end{enumerate}
\end{definition}
We will often denote holomorphic vector bundles by $\mathcal{E} = (E, \bar{\partial}_E) $. $\SL (n, \C )$-Higgs bundles form a coarse moduli space, at least once the appropriate stability conditions have been imposed. To this end, recall that the \textit{slope} $\mu$ of a complex vector bundle $E$ is the quotient
\begin{equation}
\mu (E) := \frac{\deg E}{\rk E}    
\end{equation}
of the degree of $E$ by its rank. The degree of a vector bundle is equal to the degree of its determinant line bundle, so the vector bundle underlying an $\SL (n, \C)$-Higgs bundle must have degree (and hence slope) zero.
 
\begin{definition}
A Higgs bundle $(\bar{\partial}_E, \phi)$ is called \textit{stable} if 
\begin{equation}
\mu (\mathcal{F}) <  \mu (\mathcal{E} )    
\end{equation}
for every proper, holomorphic, $\phi$-invariant subbundle $\mathcal{F} \subset \mathcal{E}$. Here, $\phi$-invariant means that $\phi(\mathcal{F}) \subset \mathcal{F} \otimes K_C$. A Higgs bundle is called \textit{polystable} if it is the direct sum of stable Higgs bundles, all of which have the same slope.
\end{definition}
Clearly, stability implies polystability. A morphism of Higgs bundles is one where the obvious diagram commutes, and stable (resp. polystable) Higgs bundles on $C$ up to isomorphism form a moduli space $\MH (C,n)$ (resp. $\MH^{ps}(C,n)$) known as the \textit{Dolbeault moduli space}, \textit{moduli space of stable Higgs bundles} or \textit{Hitchin system}. When there is no confusion we write simply $\MH := \MH(C, n)$, and similarly $\MH^{ps}$. $\MH^{ps}$ is a (singular) normal, quasiprojective variety of dimension $(n^2-1)(2g-2)$ whose smooth locus is $\MH$. Moreover, it is an algebraically completely integrable system~\cite{N1} which can be seen as follows: Note that the Higgs field $\phi$ locally on the surface is a (traceless, $K_C$-valued) endomorphism up to conjugation, so a more invariant way to describe it is in terms of its eigensystem. To this end, pick a homogeneous basis $p_2, \dots , p_n$ of the algebra of invariant polynomials $\C [\mathfrak{sl}(n,\C)]^{\SL (n, \C)}$ (e.g.\ the coefficients of the characteristic polynomial, or the polynomials $p_k (M) = \Tr(M^k)$). In the case of $\SL (n, \C)$ the polynomials will have degrees $\deg p_k = k$, and we can recall the definition of the \textit{Hitchin map}
\begin{align} \label{eq:HitchinFibration}
\begin{split}
   h : \MH (C, n) &\longrightarrow \mathcal{B} (C, n) := \bigoplus_{k=2}^n H^0(C, K_C^k),  \\
   (\bar{\partial}_E, \phi) &\longmapsto (p_2 (\phi), \dots , p_n (\phi)).
\end{split}
\end{align}
The space $\mathcal{B}=\mathcal{B}(C,n)$ is known as the \textit{Hitchin base}, and $h$ is a proper, surjective, holomorphic map whose fibers are special Lagrangian (hence generically complex, compact tori), turning $\MH^{ps}$ into an integrable system as previously advertised. %Note that $h$ does not depend on the choice of the $p_i$.

\subsection{Non-abelian Hodge Theory} \label{sec:NAHT}
Recall that a Hermitian metric $h$ on a holomorphic vector bundle $(E, \bar{\partial}_E)$ determines the \textit{Chern connection} $D$, the unique connection that is $h$-unitary and whose $(0,1)$-part agrees with the Dolbeault operator of the holomorphic bundle, $D^{(0,1)} = \bar\partial_E$. 

\begin{deff} Let $R \in \R_{>0}$. A \textit{harmonic metric with parameter $R$} for a Higgs bundle $(\bar{\partial}_E, \phi)$ is a Hermitian metric $h$ on $\mathcal{E}$ that induces the trivial metric on $\det \mathcal{E}$ and which solves the (rescaled) \textit{Hitchin's equation}
\begin{equation} \label{eq:HitchinEqn}
    F_D + R^2 \left[ \phi, \phi^{\dagger_h} \right] = 0.
\end{equation}
Here, $F_D$ is the curvature of the Chern connection $D$, and $\phi^{\dagger_h}$ denotes the adjoint with respect to the metric $h$.
\end{deff} 

A fundamental result in the theory of Higgs bundles is the celebrated \textit{Non-abelian Hodge Theorem} (NAHT), which dates back to Donaldson~\cite{DonaldsonNAHT} and Hitchin~\cite{N1} for $\SL_2$, and to Corlette~\cite{CorletteNAHT} and Simpson~\cite{simpson88} for the general case. Before we state it, recall that a flat connection $\nabla$ on a vector bundle $E$ is called
\begin{itemize}
    \item \textit{completely reducible} if for every $\nabla$-invariant subbundle $F \subset E$ there is a $\nabla$-invariant complement, and
    \item \textit{irreducible} if there are no non-trivial $\nabla$-invariant subbundles of $E$.
\end{itemize}
Irreducible (resp.\ completely reducible) flat connections up to equivalence form a moduli space $\MdR := \MdR (n,C)$ (resp.\ $\MdR^{cr}$), called the \textit{de Rham moduli space}, which is diffeomorphic (resp.\ homeomorphic) to $\MH$ (resp.\ $\MH^{ps}$) via the NAHT: 

\begin{theorem}[Non-abelian Hodge Theorem] \label{thm:NAHT}
Fix $R \in \R_{>0}$. A Higgs bundle $(\mathcal{E}, \phi)$ is polystable if and only if there exists a harmonic metric $h$ with parameter $R$. Moreover, for this $h$,
\begin{equation} \label{eq:RealTwistorLine}
    \nabla_{\zeta, R} (\bar{\partial}_E , \phi) := \zeta^{-1} R \phi + D + \zeta R \phi^{\dagger_h} 
\end{equation}
is a completely reducible flat connection for any $\zeta \in \C^\times$, which is irreducible if and only if $(\mathcal{E}, \phi)$ is stable. Conversely, a flat connection $\nabla$ is completely reducible if and only if there exists a hermitian metric $h$ such that in the unique decomposition of $\nabla$ in the form of~(\ref{eq:RealTwistorLine}) with $\zeta=1$ one has $\bar\partial_E \phi = 0$. In that case, $(\mathcal{E},\phi)$ is a polystable Higgs bundle which is stable if and only if $\nabla$ is irreducible.
\end{theorem}

Of course, $\MH$ and $\MdR$ are only diffeomorphic but not biholomorphic: While the former has compact, complex submanifolds, the latter is Stein. Nonetheless, both are K\"ahler with Riemannian metric $g$ and complex structure $I$ resp.\ $J$, and the product $K = I J = - JI$ forms another complex structure which completes that quadruple $(g, I, J, K)$ that turns $\mathcal{M}$ into a Hyperk\"ahler manifold. In particular, there is a $\CP^1$ of complex structures $I_\zeta$ which is normalized so that $I_0 = I$ and $I_1 = J$. We will mostly be concerned with complex structures $I$ and $J$ and will continue denoting the corresponding K\"ahler manifolds by $\MH$ and $\MdR$.

As with any Hyperk\"ahler manifold, it can be convenient to consider its twistor space. It will suffice for us to define the twistor space over $\C$ (rather than all of $\CP^1$), which is known as the moduli space of Deligne's $\lambda$-connections (\cite{DeligneLetter},\cite{Sim97}):

\begin{definition}
A $\lambda$-connection on $E$ is a triple $(\lambda, \delbar_E , \nabla_\lambda)$ where $\lambda \in \C$, $\delbar_E$ is a holomorphic structure on $E$ and $\nabla_\lambda : \Omega^0 (E) \rightarrow \Omega^{0,1}(E)$ is a differential operator subject to
\begin{enumerate}[label=(\roman*)]
    \item a $\lambda$-twisted Leibniz rule: for all $f \in C^\infty (C)$ and $s \in \Omega^0 (E)$ \begin{equation}
        \nabla_\lambda (fs) = \lambda \partial f \otimes s + f \nabla_\lambda(s),
    \end{equation} 
    \item a holomorphicity condition $[ \delbar_E , \nabla_\lambda ] = 0$.
\end{enumerate}
\end{definition}
When $\lambda = 0$ then $\nabla_0$ becomes $C^\infty$-linear, i.e.\ a holomorphic Higgs field. If $\lambda \neq 0$, then $\nabla := \lambda^{-1} \nabla_\lambda$ is a holomorphic connection on $(E, \delbar_E)$, or, equivalently, $D = \delbar_E + \nabla$ is a flat connection on $E$. Equivalence classes $[( \lambda, \delbar_E , \nabla_\lambda )]$ form the \textit{Hodge moduli space} $\MHod$ of (polystable) $\lambda$-connections. It comes equipped with a tautological map
\begin{equation}
    \pi: \MHod \rightarrow \C , \hspace{1cm} [( \lambda, \delbar_E, \nabla_\lambda )] \mapsto \lambda .
\end{equation}
As we just described, $\pi^{-1}(0) \simeq \MH$ and $\pi^{-1}(\lambda) \simeq \MdR$ when $\lambda \neq 0$. The family~(\ref{eq:RealTwistorLine}) is a \textit{real twistor line}: A holomorphic section of $\pi$ that obeys the reality condition
\begin{equation}
\overline{\nabla_{-\bar{\zeta}^{-1}}} = g^{-1} \nabla_\zeta g =: \nabla_\zeta . g, 
\end{equation}
where $g$ is a gauge transformation acting on connections in the usual way.

\subsection{The $\C^\times$-action}
The Hodge moduli space $\MHod$ comes equipped with a natural $\C^\times$-action 
\begin{equation} \label{eq:CstarAction}
    \xi \cdot [(\lambda, \bar{\partial}_E , \nabla_\lambda )] := [(\xi \lambda, \bar{\partial}_E, \xi \nabla_\lambda )], \hspace{1cm} \xi \in \C^\times.
\end{equation}
In particular, $\pi$ is $\C^\times$-equivariant:  %since multiplication by $\xi$ turns a $\lambda$-connection into a $(\xi \lambda)$-connection:
\begin{equation}
    \pi \left( \xi \cdot [(\lambda, \bar{\partial}_E , \nabla_\lambda )] \right) = \xi \cdot \lambda = \xi \cdot \pi [(\lambda, \bar{\partial}_E , \nabla_\lambda )]. 
\end{equation}

This $\C^\times$-action restricts to the fiber $\pi^{-1}(0)$ of Higgs bundles, where it coincides with the usual action of rescaling the Higgs field. A fixed point of the $\C^\times$-action must necessarily have $\lambda=0$ and is therefore a Higgs bundle. Such a fixed point is called a \textit{variation of Hodge structure} (VHS) and must lie in the fiber $h^{-1}(0) =: N \subset \MH$ called the \textit{nilpotent cone}. The $\xi \to 0$ limit of any Higgs bundle exists by properness of the Hitchin map and is a VHS. To illustrate that the Higgs field of a fixed point can be non-zero, let $K_C^{1/2}$ denote a choice of square root of the canonical bundle and define the following Higgs bundle:
\begin{equation} \label{eq:UniformizationPoint}
    \mathcal{E}_n = K_C^{\frac{1-n}{2}} \oplus K_C^{\frac{3-n}{2}} \oplus \dots \oplus K_C^{\frac{n-1}{2}}, \hspace{1cm} \phi_n = \begin{pmatrix} 0 & 1 & & \\  & 0 & \ddots & \\ & & \ddots & 1 \\ & & & 0 \end{pmatrix}.
\end{equation}
This Higgs bundle represents the \textit{uniformization point} in $\MH (C, n)$, and it is easy to see that
\begin{equation}
    (\mathcal{E}_n , \xi \phi_n ) = (\mathcal{E}, g^{-1}(\xi) \phi_n g(\xi)) =: (\mathcal{E}, \phi_n) . g(\xi)
\end{equation}
where
\begin{equation}
    g(\xi) = \mathrm{diag} \left( \xi^{\frac{1-n}{2}}, \xi^{\frac{3-n}{2}}, \dots, \xi^{\frac{n-1}{2}} \right).
\end{equation}
In particular, $[(\mathcal{E}_n, \xi \phi_n )] = [(\mathcal{E}_n, \phi_n)]$, so it is a fixed point of the $\C^\times$-action. More generally, by a Higgs bundle version of the Rees construction one can show that a fixed point admits a grading, meaning that it is of the following form~\cite{simpson92}:

\begin{prop} \label{prop:FixedPoints}
A Higgs bundle $[(\mathcal{E}, \phi)] \in \MH^{ps}$ is fixed under the $\C^\times$-action if and only if there is a decomposition $\mathcal{E} = \mathcal{V}_1 \oplus \dots \oplus \mathcal{V}_k$ into holomorphic subbundles with respect to which
\begin{equation}
    \phi = \begin{pmatrix} 0 & \phi_{1,2} & & \\  & 0 & \ddots & \\ & & \ddots & \phi_{k-1,k} \\ & & & 0 \end{pmatrix},
\end{equation}
where $\phi_{i,i+1} : \mathcal{V}_{i+1} \to \mathcal{V}_{i} \otimes K_C$ is holomorphic.
\end{prop}
\begin{rmk}
The Higgs bundle $(\mathcal{E}_n, \phi_n)$ sits at one extreme in the set of fixed points, namely where $k=n$. On the other end, there is the case $k=1$ where a fixed point is a pair $(\mathcal{E},0)$ and $\mathcal{E}$ is a polystable vector bundle (of rank $n$ and degree $0$).
\end{rmk}

Simpson (\cite{Sim97}, \cite{Simpson10}) has shown that the $\xi \to 0$ limit in~(\ref{eq:CstarAction}) exists (and is a VHS) for any $\lambda$-connection, and that this action defines a Bia{\l}ynicki-Birula stratification of $\MHod$. To this end, let $\mathcal{VHS} := (\MHod)^{\C^\times} = \coprod V_\alpha$ be the decomposition of the fixed locus into connected components, and let
\begin{equation}
    W_\alpha := \{ [(\lambda, \delbar_E, \nabla_\lambda )] \in \MHod : \lim_{\xi \to 0} \xi \cdot [(\lambda, \delbar_E, \nabla_\lambda )] \in V_\alpha \}.
\end{equation}
Then $\MHod = \coprod W_\alpha$ is the BB-stratification of $\MHod$, and the restrictions
\begin{align}
    W^0_\alpha := W_\alpha \cap \pi^{-1}(0), W^1_\alpha := W_\alpha \cap \pi^{-1}(1) 
\end{align}
define the \textit{Morse stratification} of $\MH$ and the \textit{partial oper stratification} of $\MdR$.

\subsection{The conformal limit} \label{sec:CL}

We will be interested in the limit of the family $\nabla_{\zeta, R} (\bar{\partial}_E , \phi)$ from (\ref{eq:RealTwistorLine}) as $\zeta$ and $R$ simultaneously approach $0$. Gaiotto~\cite{GaiottoOpersAndTBA} suggested leaving their ratio $\hbar = \zeta R^{-1}$ fixed, yielding a family of connections
\begin{equation}
    \nabla_{\hbar, R} (\bar{\partial}_E , \phi) = \hbar^{-1} \phi + D + \hbar R^2 \phi^{\dagger}, 
\end{equation}
whose $R \to 0$ limit, if it exists, is called the \textit{$\hbar$-conformal limit} of the Higgs bundle $(\bar{\partial}_E, \phi)$. Collier and Wentworth have shown that the $\hbar$-conformal limit of a Higgs bundle $[(\bar{\partial}_E, \phi)] \in \MH^{ps}$ exists if $\lim_{\xi \to 0} [(\bar{\partial}_E, \xi \phi )]$ is a \textit{stable} Higgs bundle~\cite{collier2019conformal}. In this sense, the $\hbar$-conformal limit can be viewed as a map $\mathcal{CL}_\hbar$ from a dense open set in $\MH^{ps}$ to $\MdR$, but as opposed to the NAHT map it is not continuous. Indeed, the orbit of a nilpotent Higgs bundle under the $\mathbb{C}^\times$-action completes to a holomorphically embedded $\mathbb{CP}^1$ inside of the nilpotent cone by properness of the Hitchin map. On the other hand, the image of the $\mathbb{C}^\times$-orbit in $\MdR$ cannot complete to a $\CP^1$ because $\MdR$ is Stein.

Remarkably though, $\mathcal{CL}_\hbar$ still has very nice features, one of which is the following:  Let $(\delbar_0 ,\phi_0) \in V_\alpha$ be a $\C^\times$-fixed point and define for $i \in \{0,1\}$
\begin{equation}
    W^i _\alpha (\delbar_0, \phi_0) := \{ [(\lambda, \delbar_E, \nabla_\lambda )] \in W^i _\alpha : \lim_{\xi \to 0} \xi \cdot [(\lambda, \delbar_E, \nabla_\lambda )] = [(0, \delbar_0, \phi_0)] \}.
\end{equation}
These are holomorphic subspaces of $\MH$ resp.\ $\MdR$, and the conformal limit identifies them biholomorphically~(\cite{collier2019conformal},\cite{DFKMMN}). This has one particular implication: For any $(\delbar_0, \phi_0) \in V_\alpha$, $W^0_\alpha (\delbar_0 , \phi_0)$ is biholomorphically equivalent to an affine space by BB-theory~\cite{BBtheory} (this is the algebro-geometric analog of the fact in Morse theory that the ascending manifold of a point has the topological type of a disk), and through $\CL_\hbar$ this implies that $W^1_\alpha (\delbar_0 , \phi_0)$ is biholomorphic to an affine space $\mathbb{A}^m$. Simpson~\cite{Simpson10} conjectured that the leaves $W^1_\alpha (\delbar_0, \phi_0)$ are closed inside of $\MdR$ and provided a heuristic argument: If a leaf was not closed, it could lead to a projective curve inside of $\MdR$ which is impossible because $\MdR$ is Stein\footnote{It is worth mentioning that an affine variety $X$ can contain a subvariety $Y$ that is isomorphic to $\mathbb{A}^n$ and such that $Y$ is not closed in $X$. For example, the diagonal inside of $\mathbb{CP}^1 \times \mathbb{CP}^1$ is ample and so its complement $X$ is an affine variety. Restricting to the locus where the first coordinate is nonzero gives a subvariety $Y$ that is clearly isomorphic to $\mathbb{A}^2$ and not closed in $X$.}. Note that leaves $W^0_\alpha (\delbar_0, \phi_0) \subset \MH$ are not closed when they contain non-zero nilpotent Higgs bundles that are not themselves fixed points.

%By a standard fact in algebraic geometry, an affine space inside of an affine variety\footnote{Note that $\MdR$ is an affine analytic variety, though not an affine algebraic variety.} is closed: Otherwise one could take a point in its boundary to construct a holomorphically embedded $\CP^1$, which is impossible for affine varieties. \textcolor{red}{Fix this!}

We shall be interested in the conformal limit of $\mathbb{C}^\times$-orbits of Higgs bundles, so let us elaborate a little. The conformal limit can be written out rather explicitly. Let $D$ be the Chern connection for the Higgs bundle $(\bar{\partial}_E, \phi)$, i.e.\ $D = \bar{\partial}_E + \partial_{E}$ when decomposed into types. Let $[(\bar{\partial}_0, \phi_0)] = \lim_{\xi \to 0} \xi \cdot [(\bar{\partial}_E, \phi )]$ be the corresponding fixed point with harmonic metric $h_0$ and Chern connection $D_0 = \bar{\partial}_0 + \partial_0$. Then the $\hbar$-conformal limit of $(\bar{\partial}_E, \phi)$ is given by~\cite{collier2019conformal}:
\begin{equation} \label{eq:CLCxn}
    \lim_{R \to 0} \nabla_{\hbar, R} (\bar{\partial}_E , \phi) = \hbar^{-1} \phi + \bar{\partial}_E + \partial_0 + \hbar \phi_0^{\dagger_{h_0}} =: \CL_\hbar (\bar{\partial}_E, \phi).
\end{equation}

Of course, the parameter $R$ in the family $\nabla (\hbar, R)$ originates from the usual $\C ^\times$-action on the moduli space of Higgs bundles. The following explains the interplay between this action and the parameter $\hbar$ controlling the family of conformal limits:

\begin{prop} \label{prop:ActionOnCLFamily}
Let $r \in \mathbb{R}_{>0}$. Then
\begin{equation}
    \CL_\hbar (r \cdot (\bar{\partial}_E, \phi)) = \CL_{\hbar / r} (\bar{\partial}_E, \phi).
\end{equation}
\end{prop}

\begin{proof}
Let us denote the harmonic metric with parameter $R$ by $h(R)$ and fix $r$. Then
\begin{align}
    \CL_\hbar (r \cdot (\bar{\partial}_E, \phi)) &= \lim_{R \to 0} \nabla_{\hbar, R} (r \cdot (\bar{\partial}_E , \phi)) \\
    &= \lim_{R \to 0} \hbar^{-1} r \phi + D + \hbar r R^2 \phi^{\dagger_{h(rR)}} \\
    &= \lim_{rR \to 0} (\hbar/r)^{-1} \phi + D + (\hbar/r) (rR)^2 \phi^{\dagger_{h(rR)}} \\
    &= \lim_{rR \to 0} \nabla_{\hbar/r,rR} (\bar{\partial}_E, \phi) = \CL_{\hbar / r} (\bar{\partial}_E, \phi).
\end{align}  
\hfill \end{proof}

\subsection{Very stable Higgs bundles}
In these notes, we will be interested in nilpotent Higgs bundles, particularly those that are not fixed by the $\C^\times$-action. We collect here a few of the results about their existence and structure.

Drinfeld~\cite{DrinfeldLetter} and Laumon~\cite{LaumonNilpotent} introduced the notion of a \textit{very stable vector bundle}, which is a holomorphic vector bundle on a connected, compact Riemann surface that admits no non-zero nilpotent Higgs fields. For a surface of genus $g>1$ these vector bundles are necessarily stable in the sense discussed previously, and they form a non-empty open subset in the moduli space of (stable) vector bundles~\cite{LaumonNilpotent}. Pauly and Pe\'on-Nieto \cite{PPN} proved that a stable vector bundle $[(\delbar_0, 0)]$ is very stable if and only if the leaf $W^0_\alpha (\delbar_0, 0)$ is closed inside $\MH$ or, equivalently, if the restriction of the Hitchin map to the leaf is proper. A stable bundle that is not very stable is called \textit{wobbly}.

Hausel and Hitchin~\cite{HauselHitchin} have recently generalized these definitions to the other components $V_\alpha$ of the fixed locus inside the nilpotent cone $N$.

\begin{deff}
A stable Higgs bundle $(\delbar_0, \phi_0) \in V_\alpha$ is called very stable if 
\begin{equation}
    W^0_\alpha (\delbar_0, \phi_0) \cap N = \{[(\delbar_0, \phi_0)]\}.
\end{equation}
\end{deff}

The theorem of Pauly and Pe\'on-Nieto generalizes in the sense that a stable Higgs bundle $(\delbar_0, \phi_0)$ is very stable if and only if $W^0_\alpha (\delbar_0, \phi_0) \subset \MH$ is closed. As in the case of vector bundles, a stable Higgs bundle is called \textit{wobbly} if it is not very stable.

Wobbly Higgs bundles in $\MH (C, n)$ exist for any choice of $(C,n)$ because wobbly vector bundles do, though it is an open question which components $V_\alpha$ admit wobbly Higgs bundles. Some partial progress has been made, e.g.\ for fixed points $(\delbar_0, \phi_0)$ whose underlying vector bundle is a direct sum of line bundles $L_1 \oplus \dots \oplus L_n$. In this case, $\phi_0^{n-1}$ has a single non-trivial component $b \in H^0 (C, L_1^\ast \otimes L_n \otimes K_C^{n-1})$ and $(\delbar_0, \phi_0)$ is very stable if and only if $b$ has no repeated zeroes~\cite[Theorem 1.2]{HauselHitchin}. For example, the uniformizing Higgs bundle $(\mathcal{E}_n, \phi_n)$ is evidently very stable, and because it constitutes its own connected component $V_\alpha$ this component does not contain wobbly points.

On the other hand, one might have naively expected that all components $V_\alpha$ contain very stable Higgs bundle and that they would form a dense open subset. While this is true in the case of $\SL (2, \C)$, there is a component for the $\SL (3, \C)$ moduli space for which all Higgs bundles are wobbly~\cite{Gothen}, \cite[Remark 5.12]{HauselHitchin}.

\subsection{A parametrization of the stable manifold}
Collier and Wentworth \cite{collier2019conformal} give an explicit parametrization of the leaves $W^0_\alpha (\delbar_0, \phi_0)$ which we want to quickly recall here.

First, note that any nilpotent Higgs bundle $(\delbar_E , \phi)$ induces a filtration
\begin{equation}
    0 \subset \ker \phi \subset \ker (\phi^2) \subset \dots \subset \ker (\phi^n) = \mathcal{E},
\end{equation}
and the fixed points $(\delbar_0 , \phi_0) \in V_\alpha$ are precisely those for which there is a compatible holomorphic grading as in Proposition~\ref{prop:FixedPoints}, in the sense that $\ker(\phi^j) = \mathcal{V}_1 \oplus \dots \oplus \mathcal{V}_{j}$. Such a splitting induces a $\Z$-grading on the vector bundle $\End \mathcal{E}$ via
\begin{equation}
    \End_j \mathcal{E} := \bigoplus_{i-l=j} \Hom (\mathcal{V}_i, \mathcal{V}_l),
\end{equation}
where we set $\End_j \mathcal{E} = 0$ if $|j| > k-1$, and we define
\begin{equation}
    N_+ := \bigoplus_{j > 0} \End_j \mathcal{E}, \hspace{0.5cm} N_- := \bigoplus_{j < 0} \End_j \mathcal{E}, \hspace{0.5cm} L := \End_0 \mathcal{E} \cap \mathfrak{sl}(E).
\end{equation}

Next, recall that for $(\delbar_0, \phi_0) \in V_\alpha$ with harmonic metric $h$ the connection $D = \phi_0 + \partial^h_0 + \delbar_0 + \phi_0^{\dagger_h}$ is flat, where $\partial_0^h + \delbar_0$ is the Chern connection associated to the pair $(\delbar_0, h)$. Let us define
\begin{equation}
    D' = \partial_0^h + \phi_0^{\dagger_h}, \hspace{0.5cm} D'' = \delbar_0 + \phi_0 ;
\end{equation}
then $D = D' + D''$. Define furthermore the \textit{BB-slice} through $(\delbar_0, \phi_0)$ as
\begin{equation}
    \mathcal{S}^+_{(\delbar_0, \phi_0)} := \{ (\beta, \phi_1) \in \Omega^{0,1}(N_+) \oplus \Omega^{1,0}( L \oplus N_+) | D''(\beta, \phi_1) + [\beta, \phi_1] = 0, D'(\beta, \phi_1) = 0   \}.
\end{equation}
Then the following proposition (Corollary 4.3 in~\cite{collier2019conformal}) gives a parametrization of the stable manifold $W^0_\alpha (\delbar_0, \phi_0)$ which generalizes the familiar one for the Hitchin section, the stable manifold of the uniformization point $(\mathcal{E}_n, \phi_n)$:
\begin{prop}\label{prop:ParametrizationStableMfld}
Let $(\delbar_0, \phi_0) \in V_\alpha$ be a stable VHS. Then the map
\begin{equation}
    p_H : \mathcal{S}^+_{(\delbar_0, \phi_0)} \rightarrow W^0_\alpha (\delbar_0, \phi_0), \hspace{0.5cm} (\beta, \phi_1) \mapsto \left[ (\delbar_0 + \beta, \phi_0 + \phi_1) \right]
\end{equation}
is a biholomorphism.
\end{prop}

\subsection{WKB analysis} \label{subsec:WKBreview}
The WKB method can be used to study the asymptotics of flat sections for a family of flat connections
\begin{equation}
    \nabla_\epsilon = \epsilon^{-1} \phi + D + \epsilon \psi
\end{equation}
on a holomorphic vector bundle $\mathcal{E}$ over a Riemann surface $C$ which is not necessarily compact. These families may arise as the real twistor lines associated to a Higgs bundle via Non-abelian Hodge theory, or as the family defining the conformal limit. In either case, $\phi$ is a Higgs bundle, $D$ is a connection, and $\epsilon \psi$ contains the terms that grow in positive powers with $\epsilon$. We will be interested in the case where the leading term $\phi$ is nonzero and where $\psi$ is independent of $\epsilon$.

Let $\gamma : I=[0,1] \rightarrow C$ be a parametrized, closed loop in $C$, i.e.\ $\gamma(0) = \gamma(1) = p$. After trivializing $\gamma^\ast \mathcal{E}$, flatness of a section $s$ of $\gamma^\ast \mathcal{E}$ under the pullback connection $\gamma^\ast \nabla_\epsilon$ is expressed by the equation
\begin{equation} \label{eq:FlatSection}
    \left( dt \otimes \frac{d}{dt} + \epsilon^{-1} \gamma^\ast \phi + A + \epsilon \gamma^\ast \psi \right) s(t) = 0,
\end{equation}
where $t$ is the coordinate on $I$ and $A$ is a $1$-form representing the connection $\gamma^\ast D$.

Let us focus on the case of $\SL (2, \C)$ so that $\gamma^\ast \phi$ has two eigenvalues $\pm \mu (t) dt$. Everything that follows will depend only on the homology class of $\gamma$ and hence we may assume that $\gamma^\ast \phi$ is zero-free. Moreover, we will be interested in the case that $\gamma$ is a \textit{WKB curve}, i.e.\ that 
\begin{equation}
    \Re ( \mu (t)) > 0
\end{equation}
for all $t$. Define the \textit{period along $\gamma$} to be
\begin{equation} \label{eq:Period}
    Z_\gamma := \int_0^1 \mu (t) dt, 
\end{equation}
then for a WKB curve we clearly have $\Re (Z_\gamma) > 0$. Since $\pm \mu (t)$ are distinct, we can choose a splitting $\gamma^\ast \mathcal{E} = L_- \oplus L_+$ that diagonalizes $\gamma^\ast ( \phi )$, i.e.\ $\gamma^\ast (\phi) |_{L_\pm} = \pm (\mu) id_{L_\pm}$. Let us also write the connection 1-form $A$ with respect to this decomposition as
\begin{equation}
    A = A_+ + A_- + A_o,
\end{equation}
where $A_\pm$ are $1$-forms representing flat connections on $L_\pm$, and $A_o$ represents the off-diagonal part of $A$, i.e.\ it is a section of $\left( \Hom (L_+, L_-) \oplus \Hom(L_-, L_+) \right) \otimes T^\ast I$. The WKB method gives the following
\begin{prop} \label{prop:RegularWKB}
For a WKB curve $\gamma$, the holonomy of $\nabla_\epsilon$ around $\gamma$ grows exponentially in $\epsilon^{-1}$. More precisely, we have
\begin{equation}
   \lim_{\epsilon \searrow 0} \big( \Tr \Hol_\gamma (\nabla_\epsilon) \cdot \exp (-\epsilon^{-1} Z_\gamma )  \big) = \Hol_\gamma (A_+) \in \mathbb{C}^\times,
\end{equation}
where the limit is taken over $\epsilon > 0$.
\end{prop}

Here we have implicitly fixed an identification of the fibers over the endpoints of the interval.

\begin{rmk}
The statement of Proposition~\ref{prop:RegularWKB} goes back to~\cite{GMN09}, and the proof strategy that was outlined there has recently been formalized by Mochizuki~\cite{BDHH21}.
\end{rmk}
%\newpage
%\input{Sections/Asymptotic Expansion of ODEs}
%\newpage
\section{Nilpotent WKB analysis} \label{section:NilpotentWKB}

\subsection{A secondary Higgs field}
We shall be interested in the small $\epsilon$ asymptotics of a  family of flat connections
\begin{equation} \label{eq:FamFlatCxns}
    D_\epsilon = \epsilon^{-1} \phi + D + \epsilon \psi
\end{equation}
where the leading term $\phi$ is a nilpotent Higgs field on $E$, $D$ is a connection on $E$, and $\phi$ and $\psi$ are $\End (E)$-valued $1$-forms of type $(1,0)$ resp.\ $(0,1)$. Examples of these families include the real twistor lines $\nabla_\zeta$~(\ref{eq:RealTwistorLine}) and the families arising from the conformal limit $\mathcal{CL}_\hbar$~(\ref{eq:CLCxn}). We will be using the parameter $\epsilon$ for a general family and observe that its flatness implies a set of Hitchin-like equations:

\begin{lem}

Let $D_\epsilon = \epsilon^{-1} \phi + D + \epsilon \psi$ be a family of connections. Then flatness for all $\epsilon \in \C^\times$ is equivalent to the following set of equations:
\begin{align} \label{eq:GenHitchinEqns}
    \begin{split}
    \left[ \phi \wedge \phi \right] = \left[ \psi \wedge \psi \right] &= 0 , \\
    D \phi = D \psi &= 0, \\
    F_D + \left[ \phi \wedge \psi \right] &= 0.
    \end{split}
\end{align}
\end{lem}

\begin{proof}
The proof is completely analogous to the one that relates Hitchin's equations to the flatness of the real twistor line. One readily computes
\begin{equation}
    F_{\nabla_\epsilon} = \epsilon^{-2} \left[ \phi \wedge \phi \right] + \epsilon^{-1} D \phi + F_D + \left[ \phi \wedge \psi \right] + \epsilon D \psi + \epsilon^2 \left[ \psi \wedge \psi \right],
\end{equation}
and its vanishing for all choices of $\epsilon$ is equivalent to~(\ref{eq:GenHitchinEqns}).
\hfill \end{proof}

\medskip

Notice that in particular $\delbar_E \phi = 0$ where $\delbar_E$ is the $(0,1)$-part of $D$ and thereby a holomorphic structure on $E$. We will be concerned with the scenario where $(\delbar_E, \phi)$ defines a stable, non-zero, nilpotent Higgs bundle, and will restrict our attention to the case of $G = \SL (2, \C )$. By nilpotency, $\phi$ defines a holomorphic line subbundle $\mathcal{L}_1:= \ker \phi$ of $\mathcal{E}$ and we can choose a $C^\infty$-complement $L_2$. There is a canonical choice for $L_2$ which we will make by taking it to be the orthogonal complement with respect to the harmonic metric of $(\delbar_E, \phi)$ of $\mathcal{L}_1$. Since we work in a smooth setting, we will use $L_1$ to denote the smooth bundle underlying $\mathcal{L}_1$. This splitting induces a decomposition of $D$ of the form
\begin{equation} \label{eq:CxnSplitting}
    D = A_{1} + D' + A_{-1},
\end{equation}
where $A_1 \in \Gamma (\Hom (L_2 , L_1) \otimes \Omega^1)$, $A_{-1} \in \Gamma (\Hom (L_1, L_2) \otimes \Omega^1)$, and $D' = D_1 \oplus D_2$ is the direct sum of connections on $L_1$ and $L_2$. For a fixed $\epsilon \in \C^\times$ we can pick a square-root $\epsilon^{1/2}$ which defines an automorphism on $L_1$ (resp.\ $L_2$) by multiplication by $\epsilon^{1/2}$ (resp.\ $\epsilon^{-1/2}$) and hence an automorphism of $E$ of the form
\begin{equation} \label{eq:SL2GaugeMatrix}
g(\epsilon) = \begin{pmatrix} \epsilon^{1/2} & 0 \\ 0 & \epsilon^{-1/2} \end{pmatrix}.
\end{equation}
Notice that the ambiguity in picking a square-root contributes an overall sign, which is central in $\SL (2, \C)$, so the gauge transformation is well-defined. Conjugating by this (constant on $C$) transformation one obtains a new family
\begin{equation} \label{eq:NewFamFlatCxns}
D ' _\epsilon := D_{\epsilon^2}.g(\epsilon) = \epsilon^{-1} (A_{-1} + \phi) + D' + \mathcal{O}(\epsilon).
\end{equation}
Observe that we have rescaled the original family to be $D_{\epsilon^2}$ because $D_\epsilon . g(\epsilon{1/2})$ is not single-valued in $\epsilon$.

\begin{rmk}
Notice that the connection $D'$ is diagonal, in particular its $(0,1)$-part $\delbar'_E$ is. Thus, the underlying holomorphic vector bundle is a direct sum of line bundles $\mathcal{E}' = \mathcal{L}_1 \oplus \mathcal{L}_2$, the associated graded to the filtered bundle defined by the nilpotent Higgs field $\phi$.
\end{rmk}

\begin{prop} \label{prop:NewHiggsField}
Let $\Phi = A_{-1} + \phi$. The pair $(\delbar'_E, \Phi)$ defines a stable Higgs bundle. 
\end{prop}

\begin{proof}
Flatness of the family $D'_\epsilon$ implies a set of equations similar to (\ref{eq:GenHitchinEqns}), in particular

\begin{equation} \label{eq:AnotherFlatness}
    \left[ \Phi \wedge \Phi \right] = 2 \left[ A_{-1} \wedge \phi \right] = 0.
\end{equation}
$\phi$ is strictly upper-triangular with respect to the splitting we chose, while $A_{-1}$ is strictly lower-triangular, so their commutator is a priori diagonal and non-zero. Since $\phi$ is of type $(1,0)$ and non-zero, (\ref{eq:AnotherFlatness}) implies that $A_{-1}^{(0,1)}$, the $(0,1)$-part of $A_{-1}$, must vanish. Moreover, since $D' \Phi = 0$ and $\Phi$ is of pure type $(1,0)$, this implies that $\delbar'_E \Phi = 0$. 

Lastly we need to show stability, for which we can adapt a proof from~\cite{BHR}. Any holomorphic line subbundle $\mathcal{L}$ of $\mathcal{E}'$ induces a holomorphic line bundle homomorphism
\begin{equation}
    \mathcal{L} \mapsto \mathcal{E}'/\mathcal{L}_2 \simeq \mathcal{L}_1
\end{equation}
which is identically zero if and only if $\mathcal{L} \simeq \mathcal{L}_2$. Since $\phi \neq 0$, $\mathcal{L}_2$ is not an invariant subbundle, and hence for any invariant subbundle $\mathcal{L}$ one obtains $\deg \mathcal{L} < \deg \mathcal{L}_1$, which together with $\deg \mathcal{L}_1 < 0$ by stability of $(\delbar_E, \phi)$ completes the proof. \hfill
\end{proof} \\

\begin{rmk}
One important invariant of an $\SL (2, \C)$-Higgs bundle $\Phi$ is the associated quadratic differential $q_2 = \frac{1}{2} \Tr (\Phi^2) \in H^0 (C, K_C^2)$. In the case at hand this has a very geometric meaning: Let $\partial_E$ denote the $(1,0)$-part of $D$, and consider the composite map
\begin{equation}
    Q_2 := \phi \circ \partial_E : \mathcal{L}_1 \xrightarrow{\partial_E} \mathcal{E} \otimes K_C \xrightarrow{\phi} \mathcal{L}_1 \otimes K_C^2.
\end{equation}
\end{rmk}

\begin{lem}
$Q_2$ acts by multiplication by a holomorphic quadratic differential $q_2 \in H^0(C, K_C^2)$.
\end{lem}

\begin{proof}
This follows directly from the previous proposition but it is enlightening to check it by hand. We need to prove two things: Firstly, that $Q_2$ is linear over functions $f \in \mathcal{O} (C)$. Secondly, that $Q_2$ is indeed holomorphic.

To see that $Q_2$ is linear over functions, let $f \in \mathcal{O} (C)$ and $s \in \Gamma(C, \mathcal{L})$, then
\begin{equation}
    Q_2 (f \cdot s) = \phi (\partial f \cdot s + f \cdot \partial_E s) = \partial f \cdot \phi(s) + f \cdot \phi( \partial_E s) = f \cdot Q_2(s),
\end{equation}
where we used $\phi (s) = 0$ in the last step.

Next we check holomorphicity. $\phi$ is holomorphic by (\ref{eq:GenHitchinEqns}) and so commutes with $\delbar_E$. We can then compute the $\delbar$-derivative of $Q_2 (s)$ as
\begin{align}
    \delbar_E \left( \phi \cdot \partial_E (s) \right) &= \phi \cdot \delbar_E \partial_E(s) \\ 
    &= \phi \cdot F_D (s) \\
    &= -\phi \cdot [\phi, \psi] (s) \\ 
    &= -\phi^2 (\psi(s)) - \phi \circ \psi (\phi (s)) \\
    &= 0.  
\end{align}
Here we used (\ref{eq:GenHitchinEqns}) as well as the assumptions $\phi^2 = 0$ and $\phi(s)=0$.
\hfill \end{proof}

The interesting case will be when $\Phi$ is not nilpotent even though $\phi$ was. We conjecture that this is true whenever $(\delbar_E, \phi)$ is not a fixed point of the $\C^\times$-action but will show it only for the two most important families of flat connections: Those arising from either a real twistor line or the conformal limit.

\begin{prop} \label{prop:SL2TwistorLine}
Let $(\delbar_E , \phi)$ be a stable nilpotent Higgs bundle, and denote by $\nabla_\zeta = \zeta^{-1} \phi + \delbar_E + \partial_E^h + \zeta \phi^{\dagger_h}$ the real twistor line it defines as in~(\ref{eq:RealTwistorLine}) (with $R=1$). Then the associated \textit{secondary Higgs field} $\Phi$ is nilpotent if and only if $[(\delbar_E , \phi)]$ is a $\C^\times$-fixed point.
\end{prop}

\begin{proof}
As before we use the $C^\infty$-decomposition $E = L_1 \oplus L_2$ and decompose in analogy with~(\ref{eq:CxnSplitting})
\begin{equation}
    \delbar_E = \bar{\eta}_{-1} + \delbar'_E + \bar{\eta}_1, 
    \hspace{0.5cm} \partial^h_E = \eta_{-1} + \partial'_E + \eta_1 .
\end{equation}
The fact that $L_1$ is a holomorphic subbundle implies that $\bar{\eta}_{-1} = 0$, hence we need to check whether $\eta_{-1} = A_{-1}$ vanishes or not. Notice that $\bar{\eta}_{1} =0$ if and only if $[(\delbar_E , \phi)]$ is a $\C^\times$-fixed point by the characterization in Proposition~\ref{prop:FixedPoints}. Assume that it is not a fixed point, then $\bar{\eta}_{1} \neq 0$ and consequently $\eta_{-1} \neq 0$ by unitarity of the Chern connection $D$. Conversely, if it is a fixed point then $\bar{\eta}_{1} =0$ and by unitarity of $D$ one obtains $A_{-1} = \eta_{-1} = 0$. \hfill
\end{proof} \\

\begin{prop} \label{prop:SL2CL}
Let $\big[ (\delbar_E,\phi) \big]$ be a nonzero stable, nilpotent Higgs bundle whose underlying vector bundles is stable. Let 
\begin{equation}
\mathcal{CL}_\hbar (\delbar_E, \phi) = \hbar^{-1} \phi + \delbar_E + \partial_0 + \hbar  \phi_0^\dagger
\end{equation}
be the family of flat connections defined via the conformal limit construction~(\ref{eq:CLCxn}). Then the associated \textit{secondary Higgs field} $\Phi$ is not nilpotent.
\end{prop}

\begin{proof}
Let us restrict to the case $R=1$, the case of general $R$ being completely analogous. As in the proof of Proposition~\ref{prop:SL2TwistorLine} it suffices to show that $D = \delbar_E + \partial_0$ does not preserve $\mathcal{L}_1 = \ker (\phi )$. Since $(E, \delbar_0)$ is a stable vector bundle, $\phi_0 = 0$ by Proposition~\ref{prop:FixedPoints}, and moreover $\delbar_E = \delbar_0$ by~\cite{collier2019conformal}. In particular, invoking NAHT (Theorem~\ref{thm:NAHT}) we see that $D = \delbar_E + \partial_0 = \delbar_0 + \partial_0$ is an irreducible connection because $(\delbar_0, 0)$ is stable, and has hence no invariant subbundles.
\hfill \end{proof}

\begin{rmk}
It would be desirable to prove Proposition~\ref{prop:SL2CL} without the assumption that $(E, \delbar_E)$ is stable, in particular in light of the application we give in Corollary~\ref{cor:SimpsonConj}. We have been unable to overcome the issue of determining the shape of $\partial_0$ in terms of the decomposition of $E$ coming from the Higgs field $(\delbar_E, \phi)$.
\end{rmk}

\subsection{$\C^\times$-orbits of nilpotent Higgs bundles}

Let $(\delbar_E, \phi)$ be a nilpotent Higgs bundle that is not a $\C^\times$-action fixed point. We want to study how its orbit $\zeta \cdot (\delbar_E, \phi)$ degenerates as $\zeta \to 0, \infty$. This is a Higgs bundle version of the Rees construction (see e.g.\ \cite{HauselHitchin}, Proposition 3.4) but we want to use of the gauge transformation (\ref{eq:SL2GaugeMatrix}) to make a more quantitative statement.

We use the $C^\infty$-splitting $E = L_1 \oplus L_2$ discussed before, with respect to which one can write
\begin{equation}
   \delbar_E = \begin{pmatrix} \delbar_1 & \bar{\eta}_{1} \\ 0 & \delbar_2 \end{pmatrix}, \hspace{1cm} \phi = \begin{pmatrix} 0 & \psi \\ 0 & 0 \end{pmatrix} .
\end{equation}

\noindent Then $\xi \in \C^\times$ acts as
\begin{equation}
    \begin{split}
        \xi \cdot \left( \begin{pmatrix} \delbar_1 & \bar{\eta}_{1} \\ 0 & \delbar_2 \end{pmatrix},   \begin{pmatrix} 0 & \psi \\ 0 & 0 \end{pmatrix} \right) 
            &= \left( \begin{pmatrix} \delbar_1 & \bar{\eta}_{1} \\ 0 & \delbar_2 \end{pmatrix},   \begin{pmatrix} 0 & \xi \cdot \psi \\ 0 & 0 \end{pmatrix} \right) \\
            &\sim \begin{pmatrix} \xi^{-1/2} & 0 \\ 0 & \xi^{1/2} \end{pmatrix} \cdot \left( \begin{pmatrix} \delbar_1 & \bar{\eta}_{1}\\ 0 & \delbar_2 \end{pmatrix},   \begin{pmatrix} 0 & \xi \cdot \psi \\ 0 & 0 \end{pmatrix} \right) \cdot \begin{pmatrix} \xi^{1/2} & 0 \\ 0 & \xi^{-1/2} \end{pmatrix} \\
            &= \left( \begin{pmatrix} \delbar_1 & \xi^{-1} \bar{\eta}_{1} \\ 0 & \delbar_2 \end{pmatrix},   \begin{pmatrix} 0 & \psi \\ 0 & 0 \end{pmatrix} \right).
    \end{split}
\end{equation}
In particular we see that
\begin{equation}
    \lim_{\xi \to \infty} \xi \cdot \left( \begin{pmatrix} \delbar_1 & \bar{\eta}_{1} \\ 0 & \delbar_2 \end{pmatrix},   \begin{pmatrix} 0 & \psi \\ 0 & 0 \end{pmatrix} \right) = \left( \begin{pmatrix} \delbar_1 & 0 \\ 0 & \delbar_2 \end{pmatrix},   \begin{pmatrix} 0 & \psi \\ 0 & 0 \end{pmatrix} \right) =: (\delbar_\infty, \phi_\infty),
\end{equation}
the \textit{associated graded} to $(\delbar_E, \phi)$.

If the bundle $(E, \delbar_E)$ is itself stable as a vector bundle, we also see that
\begin{equation}
    \lim_{\xi \to 0} \xi \cdot (\delbar_E , \phi) \to (\delbar_E, 0).
\end{equation}
Otherwise we have to be a little more careful and implore the parametrization from Proposition~\ref{prop:ParametrizationStableMfld}. In the rank two case this means that a fixed point $(\delbar_0, \phi_0)$ with nonzero Higgs field is of the form
\begin{equation}
    \delbar_0 = \begin{pmatrix} \delbar_1 & 0 \\ 0 & \delbar_2 \end{pmatrix}, \hspace{1cm} \phi_0 = \begin{pmatrix} 0 & \psi \\ 0 & 0 \end{pmatrix} .
\end{equation}
Any stable Higgs bundle $(\delbar_E , \phi)$ that flows to $(\delbar_0, \phi_0)$ is necessarily of the form
\begin{equation}
    \delbar_E = \begin{pmatrix} \delbar_1 & 0 \\ \bar{\eta}_{-1} & \delbar_2 \end{pmatrix}, \hspace{1cm} \phi = \begin{pmatrix} \mu & \psi \\ \chi & -\mu \end{pmatrix}.
\end{equation}
Indeed, 
\begin{equation}
    \begin{split}
        \xi \cdot \left( \begin{pmatrix} \delbar_1 & 0 \\ \bar{\eta}_{-1} & \delbar_2 \end{pmatrix},  \begin{pmatrix} \mu & \psi \\ \chi & -\mu \end{pmatrix} \right) 
            &= \left( \begin{pmatrix} \delbar_1 & 0 \\ \bar{\eta}_{-1} & \delbar_2 \end{pmatrix},   \xi \cdot \begin{pmatrix} \mu & \psi \\ \chi & -\mu \end{pmatrix} \right) \\
            &\sim \begin{pmatrix} \xi^{-1/2} & 0 \\ 0 & \xi^{1/2} \end{pmatrix} \cdot \left( \begin{pmatrix} \delbar_1 & 0 \\ \bar{\eta}_{-1} & \delbar_2 \end{pmatrix},   \xi \cdot \begin{pmatrix} \mu & \psi \\ \chi & -\mu \end{pmatrix} \right) \cdot \begin{pmatrix} \xi^{1/2} & 0 \\ 0 & \xi^{-1/2} \end{pmatrix} \\
            &= \left( \begin{pmatrix} \delbar_1 & 0 \\ \xi \bar{\eta}_{-1} & \delbar_2 \end{pmatrix},  \begin{pmatrix} \xi \mu & \psi \\ \xi^2 \chi & -\xi \mu \end{pmatrix} \right) \\
            &\xrightarrow{\xi \to 0} \left( \begin{pmatrix} \delbar_1 & 0 \\ 0 & \delbar_2 \end{pmatrix}, \begin{pmatrix} 0 & \psi \\ 0 & 0 \end{pmatrix} \right)
            .
    \end{split}
\end{equation}

In summary, the limit points $(\delbar_0, \phi_0)$ and $(\delbar_\infty , \phi_\infty)$ of a $\C^\times$-orbit have a canonical holomorphic line subbundle $\mathcal{L}_1$. If the Higgs field at this fixed point is nonzero then it is nothing but its kernel\footnote{Note that $\phi_\infty$ is never zero.}. This means that the procedure from the previous section can be canonically applied even to an entire $\CP^1 \subset N$, and the secondary Higgs bundle $(\delbar'_0, \Phi)$ associated to a fixed point agrees with the fixed point itself.

\begin{prop} \label{prop:SmoothNotHolomorphic}
Let $\mathcal{O} \simeq \CP^1 \subset \MH$ denote the closure of the $\C^\times$-orbit of a stable, nilpotent Higgs bundle that is not a fixed point of the $\C^\times$-action. Let
\begin{equation}
    \mathcal{S}: \mathcal{O} \to \MH, \hspace{1cm} (\delbar_E, \phi) \mapsto (\delbar'_E, \Phi) 
\end{equation}
denote the map that assigns the secondary Higgs bundle. Then $\mathcal{S}$ is smooth but not holomorphic.
\end{prop}
\begin{proof}
First smoothness: Locally we may view the $(\delbar_E, \phi)$ as living on a fixed unitary bundle $E$. More precisely, the Higgs bundles can be represented by pairs $(D, \phi)$ where $D$ is a unitary connection on $E$ and $\phi$ is an adjoint-valued 1-form such that the self-duality equations are fulfilled. Since the modifications we described are smooth in the components of these pairs, the resulting Higgs field $(\delbar'_E, \Phi)$ will depend smoothly on the input.

Now assume that $\mathcal{S}$ is holomorphic. Then its composition with the Hitchin map (\ref{eq:HitchinFibration}) gives a holomorphic map 
\begin{equation}
    h \circ \mathcal{S} : \CP^1 \to \mathcal{B}(C,2) = H^0 (C, K_C^2).
\end{equation}
This is necessarily constant as a map from a projective curve to an affine space. But we have seen that $\mathcal{S}(\mathcal{O})$ contains both nilpotent and non-nilpotent Higgs bundles. \hfill
\end{proof} \\

\subsection{Nilpotent WKB}
The construction from the previous section allows us to analyze the asymptotics of families of flat connections as in~(\ref{eq:FamFlatCxns}) whose leading term is a nilpotent Higgs field through a related family with regular Higgs field, as least as long as the Higgs field is not coming from a $\C$-VHS. To this end, let $\gamma : I=[0,1] \to C$ be a parametrized path. Then flatness of a covariantly constant section $s$ with respect to the pullback connection $\gamma^\ast D_\epsilon$ means that
\begin{equation}
    \left( dt \otimes \frac{d}{dt} + \epsilon^{-1} \gamma^\ast \phi + \gamma^\ast A + \epsilon \gamma^\ast \psi \right) s(t) = 0,
\end{equation}
where $t$ is a local coordinate on $I$, $d+A$ is a local trivialization of $D$, and $s(t)$ is a section of $\gamma^\ast E$. For simplicity we may assume that $\phi$ has no zeroes or poles along $\gamma$ (since the zeroes and poles of $\phi$ are discrete, it suffices to homotope $\gamma$ slightly if necessary). 

We shall be interested in the (small $\epsilon$) asymptotics of flat sections for the family $D_\epsilon$, or equivalently the family~$
D'_\epsilon$. Let us assume that $\gamma^\ast (\phi + A_{-1})$ is a regular Higgs field with non-zero eigenvalues $\pm \lambda (t)$ along $\gamma$, such as the families coming from the conformal limit or real twistor lines. Under these assumptions, we can directly apply Proposition~\ref{prop:RegularWKB}. This proves:

\begin{thm} \label{thm:WKBInRank2}
Let $\phi$ be a nilpotent Higgs field such that its associated secondary Higgs field $\Phi$ is not nilpotent. Then the holonomy of $D_\epsilon$ along a WKB curve $\gamma$ grows exponentially in $\epsilon^{-1/2}$. More precisely, one has
\begin{equation}
\lim_{\epsilon \searrow 0} \left( \Tr \Hol_\gamma (D_\epsilon ) \cdot \exp (- \epsilon^{-1/2} Z_\gamma )\right) \in \C^\times,
\end{equation}
where $\Re (\epsilon^{-1/2} Z_\gamma ) > 0$.
\end{thm}

\begin{rmk}
In light of the previous Theorem it makes sense to ask about the existence of closed WKB curves on Riemann surfaces with a quadratic differential. We show in Appendix~\ref{app:WKBcurves} that they exist on any Riemann surface with a non-zero quadratic differential using the notion of \textit{(half-)translation surfaces}. We also give a few concrete examples of WKB curves where the quadratic differential is meromorphic.
\end{rmk}

As an application of Theorem~\ref{thm:WKBInRank2}, we show that the leaves $W^1_\alpha (\delbar_0, \phi_0)$ are closed inside $\MdR$:

\begin{cor} \label{cor:SimpsonConj}
Let $(\delbar_0, \phi_0) \in V_\alpha$ be a stable VHS. Then the leaf $W^1_\alpha (\delbar_0, \phi_0) \subset \MdR$ is a closed subspace.
\end{cor}

\begin{proof}
To ease notation we write $W^i = W^i_\alpha (\delbar_0, \phi_0)$ ($i=0,1$). Recall that the $\hbar$-conformal limit of a Higgs bundle $(\delbar_E, \phi) \in W^0$ is given by
\begin{equation}
    \mathcal{CL}_\hbar (\delbar_E, \phi) = \hbar^{-1} \phi + \delbar_E + \partial_0 + \hbar \phi_0 ^{\dagger_{h_0}},
\end{equation}
and that $\mathcal{CL}_\hbar$ gives a biholomorphism $W^0 \to W^1$. Indeed, both are affine spaces and the $\C^\times$-action on the former induces a $\C^\times$-action on the latter which by Proposition~\ref{prop:ActionOnCLFamily} is for $r \in \mathbb{R}_{>0} \subset \C^\times$ given by
\begin{equation}
   \mathcal{CL}_\hbar (r \cdot (\delbar_E, \phi)) = \mathcal{CL}_{\hbar/r} (\delbar_E, \phi).
\end{equation}
Assume that $W^1$ is not closed, we claim that any boundary point $D_\infty$ must necessarily be the limit of a sequence of points on a single ray, say
\begin{equation} \label{eq:BoundaryRay}
    D_\infty = \lim_{n \to \infty} \mathcal{CL}_{\hbar/r_n} (\delbar_E, \phi),
\end{equation}
where $r_n \to \infty$ as $n \to \infty$. But Theorem~\ref{thm:WKBInRank2} and Proposition~\ref{prop:ExistenceWKBCurves} then imply that the trace of the holonomy of this sequence is unbounded, giving the desired contradiction.

It remains to show (\ref{eq:BoundaryRay}) for which we proceed as follows. Let $(D_n)_n \subset W^1$ be a sequence converging to $D_\infty$ and let $b = \mathcal{CL}_\hbar (\delbar_0, \phi_0)$ be a base point for the affine space $W^1$ such that without loss of generality $b \notin ( D_n )$. The $\C^\times$-action on $W^1$ restricts to a $\mathbb{R}_{>0}$-action that is free on $W^1 - \{ b \}$, let $\mathbb{S}$ be the quotient by this action. $\mathbb{S}$ should be viewed as the sphere of directions of rays emanating from $b$, in particular it is compact.

Let $b_n$ be the image of $D_n$ under the quotient map $q: W^1 - \{ b \} \to \mathbb{S}$, then by compactness of $\mathbb{S}$ there is a convergent subsequence $( b'_n )$ of $( b_n )$ whose limit is $b_\infty \in \mathbb{S}$. Let $(D'_n)$ be the corresponding subsequence of $D_n$, then there is a sequence $(C_n) \subset q^{-1} (b_\infty)$ such that $\lim_n C_n = \lim_n D'_n = D_\infty$.
\hfill
\end{proof} \\

\begin{comment}
\begin{rmk}
Of course, both the Theorem and its Corollary depend on the existence of a WKB curve. In general, we do not know of a satisfying criterion to check whether a given Riemann surface $C$ together with a quadratic differential $q_2$ has a WKB curve. We have constructed in Appendix~\ref{app:WKBcurves} a few families of examples using the language of half-translation surfaces. More generally speaking, the asymptotic behaviour described by Proposition~\ref{prop:RegularWKB} is expected to remain true even in the absence of a WKB curve. This is an active area of research (\cite{IwakiNakanishiI},\cite{KoikeSchafke},\cite{NikitaFuture}) and the particular statement we require has to the best of our knowledge not yet appeared in the literature.
\end{rmk}
\end{comment}

\subsection{A twistor space interpretation}
Let us take a closer look at the previous construction for the case that the family of flat connections is the twistor family
\begin{equation}
    \nabla_\zeta = \zeta^{-1} \phi + D_h + \zeta \phi^\dagger
\end{equation}
for a nilpotent Higgs field $(\delbar_E, \phi)$ that is not a $\C^\times$-fixed point. As the name suggests, this family is a real holomorphic section of the \textit{Deligne-Hitchin moduli space}, the twistor space $\MDH \to \CP^1$ of the hyperk\"ahler Hitchin system. More accurately, $\nabla_\zeta$ describes the section over $\zeta \in \C^\times$ but the section extends over $\zeta = 0$ in the sense of $\lambda$-connections, and by reality also over $\zeta = \infty$. The normal bundle to the section is isomorphic to $\mathcal{O}_{\CP ^1}(1)^d$ (where $d = \dim \MH$).

As we have just seen, associated to such a family is a new family of connections
\begin{equation} \label{eq:NotATwistorLine}
    \nabla'_\zeta = \zeta^{-1} \Phi + D' + \zeta \Psi \sim \nabla_{\zeta^2}
\end{equation}
where $D'$ is a diagonal connection and $\Psi$ is not necessarily the adjoint of $\Phi$. Indeed, (\ref{eq:NotATwistorLine}) describes a section of twistor space over $\C^\times \subset \CP^1$ and it is reasonable to ask whether the section can be extended over $0$ and $\infty$, and whether the section has the properties of a real twistor line.

The answer to the former question is affirmative by Proposition~\ref{prop:NewHiggsField}: The pair $(\delbar'_E, \Phi)$ defines a stable Higgs bundle showing that the section extends over $0$. The analogous statement holds for the pair $(\partial'_E, \Psi)$ which forms a stable Higgs bundle on $\bar{C}$, hence showing that the section extends to $\infty$. The latter question is answered in the following Proposition:
\begin{prop} \label{prop:NotATwistorLine}
The section $\nabla'_\zeta$ is not real, i.e.\ $\overline{\nabla'_{-\bar{\zeta}^{-1}}}$ and $\nabla'_\zeta$ are not gauge-equivalent.
\end{prop}
\begin{proof}
We use once again the $C^\infty$-splitting $E = L_1 \oplus L_2$ with respect to which
\begin{equation}
    \phi = \begin{pmatrix} 0 & \psi \\ 0 & 0 \end{pmatrix}, \hspace{0.5cm} D = \begin{pmatrix} D_1  & -\omega^\ast \\ \omega & D_2 \end{pmatrix}, \hspace{0.5cm} \phi^\dagger = \begin{pmatrix} 0 & 0 \\ \psi^\ast & 0 \end{pmatrix}.
\end{equation}
After the gauge transformation one obtains
\begin{equation}
    \Phi = \begin{pmatrix} 0 & \psi \\ \omega & 0 \end{pmatrix}, \hspace{0.5cm} D' = \begin{pmatrix} D_1  & 0 \\ 0 & D_2 \end{pmatrix}, \hspace{0.5cm} \Psi = \begin{pmatrix} 0 & -\omega^\ast \\ \psi^\ast & 0 \end{pmatrix}.
\end{equation}
The question is whether there is a gauge transformation $\eta(\zeta)$ such that
\begin{align} \begin{split}
   \overline{\nabla'_{-\bar{\zeta}^{-1}}} &= -\zeta^{-1} \overline{\Psi} + \overline{D'} - \zeta \overline{\Phi} \\
   &\stackrel{!}{=} \zeta^{-1} \eta^{-1}\Phi \eta + D'.\eta + \zeta \eta^{-1}\Psi \eta  = \nabla'_\zeta . \eta, \end{split}
\end{align}
or equivalently
\begin{equation}
    -\overline{\Psi} = \eta^{-1} \Phi \eta, \hspace{0.5cm} \overline{D'} = D' . \eta, \hspace{0.5cm} -\overline{\Phi} = \eta^{-1} \Psi \eta.
\end{equation}
However, this is impossible since the determinants of $\overline{\Psi}$ and $\Phi$ differ by a sign. \hfill
\end{proof} \\

\begin{rmk}
Hitchin~\cite{HitchinWrongSign} considered solutions to a signed version of the self-duality equations
\begin{align}
    \delbar_E \phi &= 0,\\
    F_D - [\phi, \phi^\dagger ] &= 0
\end{align}
Such a solution corresponds to a family of flat connections
$\zeta^{-1} \phi + D - \zeta \phi^\dagger$. Certain nilpotent Higgs bundles $(\delbar_E, \phi)$ form solutions to these equations and one could apply a similar strategy as ours to them by using a gauge transformation to produce a new section whose leading term is a non-nilpotent Higgs field. This idea was applied by Biswas, Heller and R\"oser~\cite[Theorem 3.4]{BHR} to produce a new family which they have shown to be a real twistor line. However, it fails a different property that families arising from solutions to Hitchin's equations have, which the authors have called $\tau$\textit{-negativity} (where $\tau$ denotes the involution $\zeta \mapsto -\overline{\zeta}^{-1}$ of $\CP^1$ which is part of the definition of \textit{reality}).
\end{rmk}
%\newpage
\section{The higher rank case} \label{sec:SLN}
In this section we want to generalize the previous discussion to Higgs bundles for the group $G = \SL (n, \C)$. We will restrict our attention to the case of a real twistor line arising from solution of Hitchin's equations so that the starting point is the family of flat connections
\begin{equation} \label{eq:TwistorLineSLN}
    \nabla_\zeta := \zeta^{-1} \phi + D_h + \zeta \phi^{\dagger_h}
\end{equation}
with a nilpotent $\SL (n, \C)$-Higgs field $\phi$. We will suppress the index $h$ from now on. The first difference from the $\SL (2)$-case is that the necessary gauge transformation depends on the \textit{type} of $\phi$, i.e.\ the size of its Jordan blocks. Recall that a nilpotent Higgs bundle $(\delbar_E, \phi)$ defines a filtration by holomorphic subbundles
\begin{equation} 
    0 \subset \ker \phi \subset \ker (\phi^2) \subset \cdots \subset \ker(\phi^k) = \mathcal{E},
\end{equation}
where $k \leq n$ is the smallest integer such that $\phi^k = 0$. Denote this filtration by $\mathcal{E}_i := \ker (\phi^i)$ ($i = 0, .. , k)$, and the pieces of the associated graded by $V_i := \mathcal{E}_i/\mathcal{E}_{i-1}$ ($i=1,\dots,k)$. Unless $(\delbar_E, \phi)$ is a fixed point, there is no holomorphic isomorphism $E = \oplus V_i$, so we will work with the $C^\infty$-splitting coming from picking orthogonal complements with respect to the harmonic metric. This data determines an ordered partition of $n$ via
\begin{equation}
    (\dim V_1, \dots , \dim V_k)
\end{equation}
which we call the \textit{type of} $(\delbar_E, \phi)$. Notice that its transpose partition encodes the size of the Jordan blocks of $\phi$ in decreasing order. On this note, let us recall that the Jordan normal form of $\phi$ is constant on a dense open subset of the connected Riemann surface $C$:

\begin{lem}
Let $\phi \in H^0 (C, \mathrm{ad} P \otimes K_C)$ be nilpotent. Then its Jordan type is the same on all of $C$ with the possible exception of finitely many points.
\end{lem}

\begin{proof}
It suffices to show this over a coordinate chart $U$, which we may take to be an open disk. Then $(\mathrm{ad}P \otimes K_C)|_U$ trivializes and $\phi$ defines a nilpotent matrix with values in $\C (U)$, meromorphic functions on $U$. $\phi$ can be conjugated into its Jordan normal form $\psi$, 
\begin{equation}
    \psi = g^{-1} \phi g,
\end{equation}
where $g \in \GL(n, \C (U) )$. There is a finite set of points $D \subset U$ where an entry of $g$ or $g^{-1}$ has a pole\footnote{This is true even if there is a pole at $z \in \partial U$ which one can see by going to a different chart $U'$ containing $z$ in its interior.}, so the entries of $\psi$ are holomorphic on $U - D$. Because $\psi$ is the Jordan normal form of a nilpotent matrix, the entries are also discrete, hence constant on $U-D$. \hfill \end{proof} \\

\subsection{Higgs fields of maximal type (1,1,\dots,1)} \label{subsec:MaxType}
Let us start with the case that $\phi$ is of type $(1,1,\dots,1)$, i.e.\ its Jordan form is a single Jordan block of maximal size
\begin{equation} \label{eq:MaxJNF}
     \phi = \begin{pmatrix} 0 & \phi_{1,2} & & \\  & 0 & \ddots & \\ & & \ddots & \phi_{n-1,n} \\ & & & 0 \end{pmatrix},
\end{equation}
with respect to the canonical $C^\infty$-splitting $E = \oplus V_i$, where $\phi_{i,i+1}: V_{i+1} \rightarrow V_i \otimes \Omega^1$ are nonzero. We decompose the connection $D$ with respect to this grading as 
\begin{equation} \label{eq:DecompCxn}
    D = \sum_{i,j=1}^{n} A_{i,j}
\end{equation}
where $A_{i,i}$ is a (unitary) connection on the line bundle $V_i$ while $A_{i,j}$ is a section of $\Hom (V_j, V_i) \otimes \Omega^1$ when $i \neq j$.
It will also be convenient to group these together according to their distance to the diagonal, so we let 
\begin{equation}
    A_k := \sum_i A_{i,i+k}.
\end{equation}
For example, $\nabla' := A_0$ is just the diagonal part of $\nabla$, the direct sum of connections on the line bundles $V_i$. We could similarly decompose $\phi^\dagger = \sum_{i,j} \psi_{i,j} = \sum_k \psi_k$ but in our unitary frame this simplifies to $\phi^\dagger = \psi_{-1}$. We construct a similar gauge transformation as in the rank $2$ case,
\begin{equation} \label{eq:ChangeOfFrame}
    g_n (\zeta) := \diag ( \zeta^{(1-n)/2}, \zeta^{(3-n)/2}, \dots , \zeta^{(n-1)/2} ),
\end{equation}
to construct a family related to~(\ref{eq:TwistorLineSLN}) whose leading term is a regular Higgs bundle. First, notice that (\ref{eq:ChangeOfFrame}) yields a well-defined gauge transformation: Indeed, if $n$ is odd then it only contains integer powers, while for even $n$ the ambiguity in picking a square-root contributes an overall sign to $g_n$, which is central in $\mathrm{SL}(n, \C)$. The gauge transformation acts via conjugation on all pieces of $\nabla_\zeta$ and scales the entries according to their distance from the diagonal:
\begin{equation} \label{eq:DefAk}
    A_k . g_n(\zeta) = \zeta^{k} A_k,
\end{equation}
while $\phi. g_n = \zeta \phi$ and $\phi^\dagger . g_n = \zeta^{-1} \phi^\dagger$. It therefore transforms the family as
\begin{equation}
    \nabla'_\zeta := \nabla_{\zeta^n}.g_n(\zeta) = \zeta^{1-n} (\phi + A_{1-n}) + \zeta^{2-n} A_{2-n} + \dots + \nabla' + \zeta A_1 + \dots + \zeta^{n-1} (A_{n-1} + \phi^\dagger),
\end{equation}
but the leading term is only regular if $A_{1-n} \neq 0$ which is not necessarily true. We let $m \leq n$ be the largest integer for which $A_{1-m} \neq 0$, and remark that $m > 1$ if and only if $(\delbar_E, \phi)$ is not a fixed point which follows directly from Proposition~\ref{prop:FixedPoints} and unitarity of the Chern connection. We will assume $m>1$ from now on and define
\begin{equation} \label{eq:NewFamilySLn}
    \nabla'_\zeta := \nabla_{\zeta^{m}}.g_n(\zeta) = \zeta^{1-m} (\phi + A_{1-m}) + \zeta^{2-m} A_{2-m} + \dots + \nabla' + \dots \zeta^{m-1} (A_{m-1} + \phi^\dagger) + \dots.
\end{equation}
%Notice that we have used $\nabla_{\zeta^{m}}$ so that $\nabla'_\zeta$ is single-valued in $\zeta$. In other words, we have parametrized $\nabla'$ after precomposing with $\zeta \mapsto \zeta^{m}$ on $\C^\times$ or $\CP^1$. 
As in the rank $2$ case, the constant term is diagonal and therefore its $(0,1)$-part $\delbar'_E$ defines a holomorphic structure on $E$ which turns the line bundles $E_i$ into holomorphic line subbundles. Let $\Phi := \phi + A_{1-m}$ and $\zeta_m = \exp (2\pi i/m)$, then 
\begin{equation}
    \Phi.g_n(\zeta_m) = \zeta_m \phi + \zeta_m^{1-m} A_{1-m} = \zeta_m \Phi,
\end{equation}
i.e.\ $(\delbar'_E, \Phi)$ is fixed under a cyclic subgroup of order $m$. These bundles have been studied previously~(e.g.\ \cite{Simpson09,CollierThesis}) and are called $m$-cyclic Higgs bundles:
\begin{deff}
An $m$-cyclic $\SL(n, \C)$-Higgs bundle is an $\SL(n, \C)$-Higgs bundle $(\delbar_E, \phi)$ of the special form
\begin{equation}
    \mathcal{E} = \mathcal{E}_1 \oplus \dots \oplus \mathcal{E}_k, \hspace{0.5cm} \phi = \left( \phi_{i,j} \right)_{i,j=1}^k,
\end{equation}
such that $\mathcal{E}_i$ are holomorphic vector bundles, and $\phi_{i,j} \in H^0(C, \Hom(\mathcal{E}_{j}, \mathcal{E}_i) \otimes K_C)$ such that $\phi_{i,i+1}$ ($i=1,..,k-1$) are required to be nonzero, $\phi_{m+i-1,i}$ ($i=1, .. ,k+1-m$) are allowed to be nonzero, and $\phi_{i,j} = 0$ otherwise.
\end{deff}
These fixed points under the subgroup $\mu_m := \langle \zeta_m \rangle \subset \C^\times$ are a generalization of variations of Hodge structure, in the sense that the underlying bundle still carries a $\mathbb{Z}$-grading, while the Higgs field is not necessarily strictly monotonic with respect to the grading but is allowed to form cycles of length $m$. A direct consequence of this structure is that the image of $\Phi$ under the Hitchin map is rather special: Since the only cycles are of length $m$ (or multiples thereof), the only nonzero holomorphic differentials associated to $\Phi$ (with respect to a homogeneous basis of $\C [\mathfrak{sl}(n)]^{\SL(n)}$ such as $p_k(X) = \Tr (X^k)$) are the $m$-, $2m$-, $\dots$, $\lfloor n/m \rfloor m$-differentials.
\begin{prop} \label{prop:SecondaryHiggsFieldSLn}
The pair $(\delbar'_E, \Phi)$ defines a stable Higgs bundle. If $(\delbar_E, \phi)$ is not a $\C^\times$-fixed point, then $(\delbar'_E, \Phi)$ is not nilpotent.
\end{prop}
\begin{proof}
We need to check holomorphicity and stability. The former follows from the flatness of the family~(\ref{eq:NewFamilySLn}). Indeed, the curvature $F_{\nabla'_\zeta}$ can be decomposed by orders of $\zeta$ as
\begin{align} \begin{split}
    F_{\nabla'_\zeta} = &\zeta^{2-2m} \left[ \Phi \wedge \Phi \right] + \zeta^{3-2m} \left[ \Phi \wedge A_{2-m} \right] + \zeta^{4-2m} \left( [\Phi \wedge A_{3-m}] + [A_{2-m} \wedge A_{2-m}] \right) + \dots + \\ &+ \zeta^{1-m} \left( \nabla' \Phi + \sum_{k=1}^{m-2} \left[ A_{-k} \wedge A_{k+1-m} \right] \right) + \dots . \end{split}
\end{align}
The proof proceeds by iteratively analyzing the equations in orders of $\zeta$. The first term implies $0 = [\Phi \wedge \Phi] = 2 [\phi \wedge A_{1-m}]$ and hence that $A_{1-m}$ is of type $(1,0)$ as we will show now. Using (\ref{eq:DefAk}) one computes
\begin{equation}
   0 = \left( [\phi \wedge A_{1-m}] \right)_{i,j} = \Big( \phi_i \wedge A_{i+1, i+2-m} + \phi_{i+1-m, i+2-m} \wedge A_{i, i+1-m}    \Big) \delta_{i, j+ 2 -m} ,
\end{equation}
where it is understood that any term vanishes which has a subscript that is not in $\{1, \dots, n \}$. Now we iterate in $i$: When $i<m$, $i+1-m \leq 0$ and only the first term contributes, showing that $A_{2, 3-m}, \dots, A_{m, 1}$ are of type $(1,0)$. When $i=m$ the second term is of type $(2,0)$ and vanishes for degree reasons, so the equation implies that $A_{m+1, 2}$ is of type $(1,0)$. Continuing this way shows that all of the components of $A_{1-m}$ are of type $(1,0)$, and then so is $\Phi$.

Vanishing of the next term in the $\zeta$-expansion similarly implies that $A_{2-m}$ is of type $(1,0)$, and continuing in this fashion proves that $A_{-k}$ is of type $(1,0)$ for all $1 \leq k \leq m-1$. Geometrically, the vanishing of the $(0,1)$-pieces of $A_{-k}$ for positive $k$ reflects the fact that $\ker (\phi^j)$ are holomorphic subbundles of $(E, \delbar_E)$. It follows that the bracket terms in the $\zeta^{1-m}$-constituent of $F_{\nabla'_\zeta}$ are of type $(2,0)$ and vanish for degree reasons, so $\nabla' \Phi = 0$ and the $(1,1)$ piece demands that $\delbar'_E \Phi = 0$.

For stability, we observe the following: If the underlying vector bundle of a Higgs bundle is the direct sum of holomorphic line bundles $\mathcal{E}_i$ such that $\sum_{i=1}^{k} \deg \mathcal{E}_i < 0$ for all $k \in \{1, \dots ,n\}$, and if the entries of the Higgs field in $\Hom (\mathcal{E}_{i+1}, \mathcal{E}_i) \otimes K_C$ ($i=1,\dots,n-1)$ are nonzero, then the Higgs bundle is stable. Indeed, stability is an open condition that is invariant under the $\C^\times$-action, and the above shows that the $\zeta \to 0$ limit of $\zeta \cdot (\delbar'_E, \Phi)$ is stable. The requirements are fulfilled for us because the original Higgs bundle $(\delbar_E, \phi)$ is stable, and $\oplus_{i=1}^k \mathcal{E}_i = \ker (\phi^k)$ is a $\phi$-invariant subbundle, hence of negative degree.

Lastly, the fact that $\Phi$ is not nilpotent follows directly from the construction: Since $A_{1-m} \neq 0$, there must be some nonzero entry $A_{i+m-1,i}$ and the composition
\begin{equation}
    \mathcal{V}_i \xrightarrow[]{A_{i+m-1, i}} \mathcal{V}_{i+m-1} \otimes K_C \xrightarrow[]{\phi^{m-1}} \mathcal{V}_i \otimes K_C^{m}
\end{equation}
is a nonzero, holomorphic morphism $M$ of line bundles. In particular, $M^k \neq 0$ is the $(i,i)$ entry of $\Phi^{m\cdot k}$ for all $k$, so $\Phi$ is not nilpotent. \hfill
\end{proof} \\

As in the rank two case, the section $\nabla'_\zeta$ is not real and consequently no real twistor line (notice that $\nabla''_\zeta$ is not single-valued and hence in particular not holomorphic in $\zeta$):
\begin{prop}
The section $\nabla'_\zeta$ is not a real section, i.e.\ it is not gauge-equivalent to $\overline{\nabla'}_{-\bar{\zeta}^{-1}}$.
\end{prop}
\begin{proof}
Recall that
\begin{equation}
    \nabla'_\zeta = \zeta^{1-m} (\phi + A_{1-m}) + \dots + \zeta^{m-1} (\phi^\dagger + A_{m-1})
\end{equation}
and by unitarity of the Chern connection we have $A_{m-1} = -A_{1-m}^\dagger$. A quick computation shows that
\begin{equation}
    \overline{\nabla'}_{-\bar{\zeta}^{-1}} = (-1)^{m-1} \zeta^{m-1} (\overline{\phi} + \overline{A_{1-m}}) + \dots + (-1)^{m-1} \zeta^{1-m} \left( \overline{\phi}^\dagger - \overline{A_{1-m}}^\dagger \right),
\end{equation}
so reality would imply in particular that $(\phi + A_{1-m})$ is gauge equivalent to 
\begin{equation}
    (-1)^m \left( \overline{\phi}^\dagger - \overline{A_{1-m}}^\dagger \right) = (-1)^m (\phi^T - A_{1-m}^T).
\end{equation}
The easiest invariant at our disposal is the $m$-differential associated to the invariant polynomial $\Tr (X^m)$, which is readily computed as
\begin{align} \begin{split}
    \Tr \left( \left[ (-1)^{m-1} (\phi^T - A_{1-m}^T) \right]^m \right) &= (-1)^{m(m-1)} m \Tr (\phi^{m-1} \circ (-A_{1-m})) \\ &= - m \Tr (\phi^{m-1} \circ A_{1-m}) = -\Tr \left( (\phi + A_{1-m})^m \right), \end{split}
\end{align}
hence the two cannot be gauge equivalent. \hfill
\end{proof} \\

\begin{rmk}
It would be desirable to find a more intrinsic characterisation of the integer $m$ that is associated to a nilpotent Higgs bundle in this way. More generally, it would be interesting to know in which connected components of $N - (\MH)^{\C^\times}$, the complement of the Variations of Hodge structure inside the nilpotent cone, the Higgs bundles that attain the theoretical maximum for the integer $m$ form a dense open subset. Note that for nilpotent Higgs bundles of type $(1,\dots, 1)$ that maximum is $n$, while for other types it is a strictly smaller integer as we will explain in the next section. It is relatively straightforward to see that a maximum is achieved on an open subset but it is less clear that this number coincides with the theoretical maximum.
\end{rmk}

Lastly, let us comment on the WKB theory for these families. Note that because $\Phi$ is a $\mu_m$-fixed point, the leading term of the family
\begin{equation}
    \nabla''_\zeta := \nabla_{\zeta}.g_n(\zeta) = \zeta^{(1-m)/m} \Phi + \zeta^{(2-m)/m} A_{2-m} + \dots = \nabla'_{\zeta^{1/m}}
\end{equation}
is single-valued in $\zeta$ up to equivalence, though the subleading terms may be multi-valued. Since (at least naively) the asymptotic behaviour of holonomies of families of flat connections is controlled by their leading constituent, there is hope in proving a WKB result analogous to that in~\ref{thm:WKBInRank2}. However, the existence of a WKB curve $\gamma: [0,1] \to C$ is even more restrictive in this case: For a fixed $\mathrm{Arg} (\zeta)$ and a fixed $m$-th root $\zeta^{1/m}$, it is required that the eigenvalues $\mu_i (t) dt$ of $\gamma^\ast \Phi$ are distinct with well-ordered real parts 
\begin{equation} \label{eq:WKBConditionSLN}
\Re (\zeta^{(1-m)/m} \mu_1 (t) ) > \Re (\zeta^{(1-m)/m} \mu_2 (t) ) > \dots > \Re (\zeta^{(1-m)/m} \mu_n (t) )    
\end{equation}
for all $t$. We are unaware of a general existence theorem for WKB curves on a Riemann surface with a higher rank Higgs bundle, see also Appendix~\ref{app:WKBcurves}. Nonetheless, if a WKB curve exists then it is expected that the analogous statement to Proposition~\ref{prop:RegularWKB} and hence to Theorem~\ref{thm:WKBInRank2} should hold, namely that asymptotically
\begin{equation} \label{eq:WKBHigherRank}
    \Tr \Hol_\gamma (\nabla_\zeta ) = \Tr \Hol_\gamma (\nabla''_\zeta ) \sim \exp (\zeta^{(1-m)/m} Z_\gamma),
\end{equation}
where $Z_\gamma = \int_0 ^1 \mu_1(t) dt \neq 0$. More generally one would expect that for any loop $\gamma$ there is a constant $Z_\gamma$ such that~(\ref{eq:WKBHigherRank}) holds, though this requires the more difficult theory of \textit{higher rank exact WKB analysis} which is to the best of our knowledge currently still under development.

\subsection{For general nilpotent Higgs fields}
In this subsection we will study the case that the Higgs field is of type $(n_1, \dots, n_k)$ where $n_1 \geq \dots \geq n_k$ and $\sum n_i = n$. In terms of the Jordan normal form this means that there are $n_k$ blocks of size at least $k$. The general idea is to apply the procedure from the previous subsection ``blockwise" and for this we start by describing a convenient frame for our computation.

As before, $(\delbar_E, \phi)$ defines holomorphic subbundles $\mathcal{E}_j := \ker \phi^j$ of dimension $n_1 + \dots + n_j$, and we can pick an orthogonal splitting $\mathcal{E}_j = V_1 \oplus \dots \oplus V_j$ using the associated graded $V_i = \mathcal{E}_i / \mathcal{E}_{i-1}$. We want to further decompose the $V_i$ into line bundles for which we need the following lemma.

\begin{lem} \label{lem:OrthSplit}
Any hermitian vector bundle $(V,h)$ on a smooth, connected 2-manifold admits an orthogonal decomposition into complex line bundles $V = L_1 \oplus \dots \oplus L_p$, where $p = \rk V$.
\end{lem}
\begin{proof}
First, any vector bundle $V$ of rank at least $2$ on a surface $C$ admits a line subbundle. Indeed, it is enough to construct a nowhere vanishing section $s$ of $V$ which is possible by first taking any section with discrete zeroes and then perturbing it smoothly to avoid zeroes. The latter is possible because $rk(V) > \dim (C)$.

Given any line subbundle $L_1 \subset V$, the hermitian metric $h$ defines an orthogonal complement $V_1$. Next, pick a line subbundle $L_2 \subset V_1$ which gives an orthogonal decomposition $V_1 = L_2 \oplus V_2$ etc. Iteration of this process yields the desired orthogonal decomposition $V = L_1 \oplus \dots \oplus L_p$. \hfill
\end{proof} \\

We can apply the previous Lemma to $V_k$ to obtain an orthogonal decomposition
\begin{equation}
    V_k = L_k^{(1)} \oplus \dots \oplus L_k^{(n_k)}
\end{equation}
into complex line bundles. Next, $k > 1$ since $\phi \neq 0$ and therefore $\phi$ does not annihilate the line bundles $L_k^{(i)}$, so we can inductively define line subbundles $L_{j}^{(i)}$ ($i=1,.., n_k; j= 1,.., k -1)$ of $V_{k-1}$ via
\begin{equation}
   \phi (L_{j+1}^{(i)}) = L_j^{(i)} \otimes \Omega^1.
\end{equation}
So far we described a decomposition of the subbundles corresponding to the largest Jordan blocks, namely those of size $k$. Let $k_1 < k$ be the largest integer such that $n_{k_1} > n_k$, we use the hermitian metric to define an orthogonal complement of $L_{k_1}^{(1)} \oplus \dots \oplus L_{k_1}^{(n_k)} \subset V_{k_1}$ which we orthogonally split into line bundles using Lemma~\ref{lem:OrthSplit} to obtain
\begin{equation}
    V_{k_1} = L_{k_1}^{(1)} \oplus \dots \oplus L_{k_1}^{(n_k)} \oplus L_{k_1}^{(n_k + 1)} \oplus \dots \oplus L_{k_1}^{(n_{k_1})}.
\end{equation}
As previously, if $k_1 > 1$ then $\phi$ does not annihilate these line bundles and one can inductively define
\begin{equation}
    L_j^{(i)} = \phi (L_{j+1}^{(i)}) \otimes \left( \Omega^1 \right)^{-1},
\end{equation}
where $i = n_k +1,..,n_{k_1}$, and $j = 1,..,k_1-1$. We can continue in the same vein to orthogonally decompose the entire bundle $E$ into line bundles, which proves the following

\begin{prop}
Given a nilpotent Higgs field $(\mathcal{E}, \phi)$ of type $(n_1,..,n_k)$ and a hermitian metric $h$, there is an orthogonal decomposition into complex line bundles
\begin{equation} \label{eq:BundleGrading}
    E = \bigoplus_{\substack{j=1,\dots, k, \\ i_j = 1, \dots, n_j}} L_j ^{(i_j)}
\end{equation}
such that $\phi ( L_1 ^{(i)} ) = 0$ and for $j \neq 1$,
\begin{equation}
    \phi (L_j ^{(i)}) \subset L_{j-1}^{(i)} \otimes \Omega^1 .
\end{equation}
Moreover, the vector bundles
\begin{equation}
    \mathcal{E}_r = \ker (\phi^r) = \bigoplus_{\substack{j=1,..,r, \\ i_j = 1,..,n_j}} L_{j}^{(i_j)} 
\end{equation}
are holomorphic subbundles of $\mathcal{E}$ for $r=1,..,k$.
\end{prop}

The previous proposition shows that there is a global (over $C$) frame in which $\phi$ exhibits its Jordan canonical form. To be explicit, let $(k_1 , \dots , k_l)$ be the partition of $n$ that is transpose to $(n_1, \dots, n_k)$ and let
\begin{equation} \label{eq:FiltrationWi}
    W^{(i)} := \oplus_{j=1}^{k_i} L^{(i)}_j.
\end{equation}

Then $\phi$ restricts to a nilpotent morphism $\phi^{(i)}: W^{(i)} \to W^{(i)} \otimes \Omega^1$ that is of type $(1,1,\dots,1)$. Notice that unless $(\delbar_E, \phi)$ is a fixed point, the $W^{(i)}$ are not subbundles in the holomorphic sense. We can copy the strategy from Section~\ref{subsec:MaxType} to each block $(W^{(i)}, \phi^{(i)})$ to construct a gauge transformation $g(\zeta)$ of $\nabla_\zeta$. Denote by $m_i \in \{2, \dots, k_i \}$ the integer $m$ in the block $i$, and by $A^{(i)}_{1- m_i} \neq 0$ the part of the Chern connection restricted to $W^{(i)}$ that has weight $1-m_i$ with respect to the grading~(\ref{eq:FiltrationWi}) of $W^{(i)}$. Then $g(\zeta)$ is of the form
\begin{equation}
    g(\zeta ) = \diag \left( \zeta^{(1-k_1)/2m_1}, \zeta^{(3-k_1)/2m_1}, \dots , \zeta^{(k_1 - 1)/2 m_1}, \zeta^{(1-k_2)/2m_2}, \dots , \zeta^{(k_l - 1)/2m_l} \right) 
\end{equation}
and if we let $m := \max \{ m_i \}$ then the gauge-transformed connection takes the form
\begin{equation} \label{eq:FinalFamily}
    \nabla'_\zeta := \nabla_\zeta . g(\zeta) = \zeta^{(1-m)/m} \Phi + \dots + \delbar '_E + \dots,
\end{equation} 
where 
\begin{equation}
Phi = \sum_{i: m_i = m} \phi^{(i)} + A^{(i)}_{1-m_i}
\end{equation}
and the pair $(\delbar'_E, \Phi)$ defines a stable Higgs bundle that is not nilpotent. 

Let us comment a little more on the shape of the Higgs bundle $(\delbar'_E , \Phi)$. $\delbar'_E$ is diagonal with respect to the grading (\ref{eq:BundleGrading}). $\Phi$ is block-diagonal with square blocks $\Phi^{(i)}$ of size $k_1 < \dots < k_l$, where $(k_1, \dots , k_l)$ is the transpose partition to $(n_1, \dots, n_k)$. $\Phi^{(i)} = \phi^{(i)} + A^{(i)}_{1-m_i} $ is the sum of two pieces, where $\phi^{(i)}$ is a nilpotent Higgs field of rank $k_i$ and type $(1, \dots, 1)$ and the only non-zero components of $A^{(i)}_{1-m_i} $ have distance $m_i - 1$ from the diagonal. In particular, $\Phi^{(i)}$ is the Higgs field of an $m_i$-cyclic Higgs bundle.

This proves the following

\begin{thm} \label{thm:HigherRank}
Let $(\delbar_E, \phi)$ be a nilpotent Higgs bundle of type $(n_1, \dots, n_k)$, $n_k \neq 0$, which is not a $\C^\times$-fixed point. Then there exists an integer $m \in \{2, \dots, k \}$ and a gauge transformation $g(\zeta)$ such that (\ref{eq:FinalFamily}) holds and the pair $(\delbar'_E, \Phi)$ defines a stable Higgs bundle that is not nilpotent.
\end{thm}
%\newpage
\section{The case of parabolic Higgs bundles} \label{sec:Parabolic}
In this section we will describe how the constructions described thus far carry over to the setting of parabolic Higgs bundles. After reviewing some basic facts about parabolic Higgs bundles and their moduli spaces, a subject pioneered by Hitchin (\cite{N1}, \cite{N2}), we will turn our attention to the ``toy model'' which is a moduli space of Higgs bundles of the minimal possible real dimension, 4. In this section we restrict our attention to the case of $\SL (2, \C )$ and degree zero bundles.

\subsection{Parabolic Higgs bundles}

Fix a compact Riemann surface $C$ with a divisor $D = p_1 + \dots + p_n$, together with a pair $\{ -\rho_i , \rho_i \}$ of parabolic weights where $0 < \rho_i < 1/2$ for each puncture $p_i$. Fix furthermore a complex vector bundle $E \rightarrow C$ of rank $2$ and degree zero and an isomorphism $\mathcal{O}_C \rightarrow \det E$ which induces a holomorphic structure $\bar{\partial}_{\det E}$ on the determinant bundle.

A parabolic $\mathrm{SL}(2, \C)$- Higgs bundle on $(C, D)$ is a triple $(\bar\partial_E , \mathcal{F}_i, \phi) $ consisting of
\begin{itemize}
    \item a holomorphic structure $\bar\partial_E$ on $E$ which induces the fixed holomorphic structure $\bar\partial_{\det E}$ on its determinant line bundle,
    \item a complete flag $\mathcal{F}_i = F_i^\bullet = \{0 = F_i^0 \subset F_i^1 \subset F_i^2 = E_{p_i} \} $ at each of the punctures,
    \item  a holomorphic section $\phi \in H^0 (C, \End ( (E, \bar\partial_E)) \otimes K_C (D) )$ which is traceless and such that $\mathrm{res} \; \phi (p_i) : F^j_i \rightarrow F^{j-1}_i \otimes K_C (D)_{p_i}, $ i.e.\ the residue of the Higgs field is strictly upper triangular\footnote{Higgs fields with this condition are often called \textit{strongly parabolic} but we will call them parabolic for simplicity as we will not deal with the weakly parabolic case here.} with respect to the filtration.
\end{itemize} 

We will denote by $\mathcal{E}$ the holomorphic vector bundle $(E, \bar\partial_E)$ together with the family of flags $\{\mathcal{F}_i \}$, so a Higgs bundle is again a pair $(\mathcal{E}, \phi)$. Its \textit{parabolic degree} is defined as
\begin{equation}
    \mathrm{pdeg} \; \mathcal{E} := \deg \mathcal{E} + \sum (\rho_i - \rho_i) = 0
\end{equation}
and so automatically vanishes. For a holomorphic line subbundle $\mathcal{L}$ of $\mathcal{E}$ we let $\alpha_i = \rho_i$ if $\mathcal{L}_{p_i} = F^1_i$, and $\alpha_i = -\rho_i$ otherwise. The parabolic structure of $\mathcal{E}$ induces a parabolic structure on $\mathcal{L}$ and the parabolic degree of $\mathcal{L}$ is
\begin{equation}
    \mathrm{pdeg} \; \mathcal{L} = \deg \mathcal{L} + \sum \alpha_i.
\end{equation}
A line subbundle $\mathcal{L}$ of $\mathcal{E}$ is called $\phi$\textit{-invariant} if $\phi (\mathcal{L}) \subset \mathcal{L} \otimes K_C(D)$. We call a Higgs bundle $(\mathcal{E}, \phi)$ 
$\vec\rho$-\textit{stable} (resp.\ $\vec\rho$-\textit{semistable}) if for every $\phi$-invariant line subbundle $\mathcal{L}$ we have
\begin{equation} \label{eq:stability}
    \mathrm{pdeg} \; \mathcal{L} = \frac{\mathrm{pdeg} \; \mathcal{L}}{\mathrm{rank} \mathcal{L}} < (\leq) \frac{\mathrm{pdeg} \; \mathcal{E}}{\mathrm{rank} \mathcal{E}} = 0.
\end{equation}

An isomorphism of parabolic Higgs bundles is an isomorphism of holomorphic vector bundles that preserves the flag structure and commutes with the Higgs field. The moduli space $\mathcal{M}_H$ of isomorphism classes of $\vec\rho$-stable Higgs bundles is a smooth, quasi-projective variety that has been constructed through geometric invariant theory in~\cite{Yoko}. It comes equipped with the usual \textit{Hitchin map} to a vector space of quadratic differentials
\begin{align} \begin{split}
    h : \mathcal{M}_H &\rightarrow \mathcal{B} = H^0 (C, K_C^2 (D)) \\
    \phi &\mapsto \Tr (\phi^2), \end{split}
\end{align}
a proper map whose generic fiber is a compact complex torus. Note that since $\phi$ has at most simple poles along $D$ and since its residue is nilpotent, $\Tr (\phi^2)$ will also have at most simple poles. 

\subsection{A secondary Higgs bundle: The parabolic case}
We want to describe how parabolic Higgs bundles give rise to families of flat connections through a parabolic version of the Non-Abelian Hodge Correspondence. For this, we must first recall what it means for a hermitian metric and a flat connection to be adapted to the parabolic structure. Locally, away from the punctures, nothing changes compared to the case discussed previously and it will suffice for us to describe the situation in a neighborhood of a puncture, cf.~\cite{BiquardGP}.

Near the puncture $p$ with parabolic structure ($\mathcal{F}^\bullet_p,  \pm \rho )$ we fix a local coordinate $z$ such that $z=0$ corresponds to $p$. A hermitian metric $h$ is called \textit{adapted to the parabolic structure at p} if there is a local holomorphic decomposition $\mathcal{E} = \mathcal{L}_1 \oplus \mathcal{L}_2$ compatible with the filtration, i.e.\ $\mathcal{L}_1 |_p = \mathcal{F}^1_p$, such that $h$ is locally of the form
\begin{equation}
    h = \begin{pmatrix} |z|^{2 \rho} & \\ & |z|^{-2 \rho}    \end{pmatrix}.
\end{equation}
$h$ is called \textit{adapted to the parabolic structure}, or just \textit{adapted}, if it is adapted to the parabolic structure at every puncture $p$. The associated Chern connection locally takes the form 
\begin{equation} \label{eq:ParChernCxn}
    D_h = h^{-1} \partial h = d + \frac{1}{z} \begin{pmatrix} \alpha & \\ & - \alpha \end{pmatrix} dz.
\end{equation}
Passing from a holomorphic frame $(e_1 , e_2)$ to a unitary one via $(\Tilde{e}_1, \Tilde{e}_2) = (|z|^{-\rho} e_1, |z|^\rho e_2)$ this becomes
\begin{equation}
    D_h = d + \begin{pmatrix} i\rho & \\ & -i \rho \end{pmatrix} d\theta
\end{equation}
for $z=r e^{i \theta}$. (Part of) the Non-Abelian Hodge Theorem then states that for every stable parabolic Higgs bundle $(\delbar_E, \phi)$ there exists a unique adapted hermitian metric $h$ such that
\begin{equation}
    \nabla_\zeta = \zeta^{-1} \phi + D_h + \zeta \phi^{\dagger_h}
\end{equation}
is a flat connection for all $\zeta \in \C^\times$ \cite{SimpsonParabolicNHC}.

Now let us restrict our attention to the case that $\phi$ is non-zero and nilpotent. This induces a holomorphic filtration $0 \subset \mathcal{K} \subset \mathcal{E}$ as before, via $\mathcal{K} = \ker \phi$. If $\phi$ has a pole at $p_i \in D$, then by definition we necessarily have $\mathcal{K}_{p_i} = F^1_i$, in other words the two flags necessarily coincide over the puncture. If on the other hand $\phi$ is regular at a puncture $p_j \in D$ then there is no compatibility condition between the two flags.

We are now in the position to take a closer look at the effect of the gauge transformation $g(\zeta) = \diag (\zeta^{1/4}, \zeta^{-1/4})$. Decomposing $D = D_h = A_1 + D' + A_{-1}$ as in~(\ref{eq:CxnSplitting}) yields
\begin{equation}
    \nabla'_\zeta := \nabla_\zeta . g(\zeta) = \zeta^{-1/2} (A_{-1} + \phi) + D' + \mathcal{O}(\zeta^{1/2}).
\end{equation}
The discussion from section~\ref{section:NilpotentWKB} carries over verbatim as long as we can ensure that the pair $(\delbar'_E, A_{-1} + \phi)$ has the correct behaviour near the punctures to define a parabolic Higgs field. But (\ref{eq:ParChernCxn}) implies that $A_{-1}$ has no poles, so $\Phi = A_{-1} + \phi$ has the same residues as $\phi$. This proves the following

\begin{thm}
Let $(\delbar_E, \phi) \in \MH$ be a $\vec \rho$-stable nilpotent parabolic $\SL_2 (\C )$-Higgs bundle that is not a fixed point of the $\C^\times$-action. Then the real twistor line 
\begin{equation}
    \nabla_\zeta = \zeta^{-1} \phi + \delbar_E + \partial_E^h + \zeta \phi^{\dagger_h}
\end{equation}
is gauge-equivalent for $\zeta \in \C^\times$ to the family
\begin{equation}
    \nabla' _\zeta = \zeta^{-1/2} \Phi + \delbar'_E + \partial'_E + \mathcal{O}(\zeta^{1/2}),
\end{equation}
and furthermore the pair $(\delbar'_E, \Phi)$ defines a $\vec \rho$-stable parabolic Higgs bundle that is not nilpotent.
\end{thm}
%\newpage
\section{The toy model} \label{sec:ToyModel}
In this section we study the so-called \textit{toy model}, the moduli space of parabolic $\SL (2, \C)$-Higgs bundles over a four-punctured sphere, which has the minimal real dimension possible, namely four. We are particularly interested in the nilpotent cone, i.e.\ the fiber $h^{-1} (0)$ of the Hitchin fibration. The non-singular part has been described in great detail, see  e.g.~\cite{FMSW}. We have not found the details below spelled out in detail in the literature and found them useful as a testing ground for the project carried out here.

Let $C = \CP^1$ and $D = p_1 + \dots + p_4$ for $p_1=0$, $p_2 = 1$, $p_3=\infty$ and $p_4 = p$, a configuration that can always be achieved through an appropriate M\"obius transformation. We start with the trivial vector bundle $E = \underline{\C}^2 \to C$, which endowed with a complex structure $\bar\partial_E$ will be isomorphic to $\mathcal{O}(-n) \oplus \mathcal{O}(n)$ for some non-negative integer $n$. Notice that in our conventions of $| \rho_i | < 1/2$ we find for a line subbundle $\mathcal{L}$ of degree $k$ that 
\begin{equation}
    k-2 < \mathrm{pdeg} \; \mathcal{L} < k+2
\end{equation}
and hence any $\phi$-invariant line subbundle of degree at least 2 would automatically destabilize a Higgs bundle $(\mathcal{E}, \phi)$ by~(\ref{eq:stability}). A complete flag $\mathcal{F}_i$ at each puncture is given by the choice of a line $L_i$. We will characterize nilpotent parabolic Higgs bundles in 2 cases, depending on whether the Higgs field vanishes or not. \\

\textbf{Case 1:} $\phi = 0$.

In this case any subbundle is $\phi$-invariant and the stability condition reduces to that of parabolic vector bundles, hence we have the two subcases $n=0$ and $n=1$.

If $\mathcal{E} = \mathcal{O}(0) \oplus \mathcal{O}(0)$ we can trivialize the bundle. Let us first note that if all 4 flags $L_i$ were to coincide, the degree $0$ line bundle passing through them would have positive parabolic degree $\sum \rho_i$ and hence destabilize $(\mathcal{E}, \phi)$. If three flags were to coincide (say $L_i = L_j = L_k \neq L_l$), the two degree zero line bundles passing through either line would have parabolic degree $\pm (\rho_i + \rho_j + \rho_k - \rho_l)$ and hence one would not be strictly negative. Similarly, one can rule out that there are two pairs of lines that agree. In summary, we can find a basis for $\mathcal{O} \oplus \mathcal{O}$ in which the lines take the form 
\begin{equation} \label{eq:4lines}
L_1 = \C \begin{pmatrix} 0 \\ 1 \end{pmatrix}, \; L_2 = \C \begin{pmatrix} 1 \\ 1 \end{pmatrix}, \; L_3 = \C \begin{pmatrix} 1 \\ 0 \end{pmatrix}, \; L_4 = \C \begin{pmatrix} w \\ 1 \end{pmatrix},
\end{equation}
for some $w \in \C \cup \{ \infty \}$ by an appropriate coordinate transformation (and, if necessary, relabelling of the punctures $p_i$).

Next, note that a holomorphic line subbundle must have non-positive degree and a line subbundle of degree $\leq -2$ automatically has negative parabolic degree. For the degree $-1$ and degree $0$ line subbundles, we must take a closer look at the flag structure as well as the parabolic weights. If $w \notin \{ 0, 1, \infty \}$, a degree zero line subbundle can hit at most one of the flags and will not destabilize as long as 
\begin{equation}
\rho_{\sigma(1)} < \rho_{\sigma(2)} + \rho_{\sigma(3)} + \rho_{\sigma(4)}    
\end{equation}
for any permutation $\sigma \in S_4$. Similarly, if $w \in \{ 0,1,\infty \}$ then there is a unique degree zero line subbundle which meets two of the flags, and for it to not destabilize $(\mathcal{E} , \phi)$ requires $\rho_1 + \rho_4 < \rho_2 + \rho_3$ if $w=0$, and analogously for the other two cases. If $w \neq p$ then a degree $-1$ line subbundle can meet at most 3 flags (the tautological bundle for example meets the flags over $0$, $1$, and $\infty$), so stability requires
\begin{equation}
1+ \rho_{\sigma(1)} > \rho_{\sigma(2)} + \rho_{\sigma(3)} + \rho_{\sigma(4)}    
\end{equation}
for all permutations $\sigma \in S_4$. If on the other hand $w = p$, then the tautological bundle meets all 4 lines and $(\mathcal{E}, \phi)$ is unstable unless $\rho_1 + \rho_2 + \rho_3 + \rho_4 <1$. From now on, we will assume that all of these conditions are fulfilled, i.e. 
\begin{align}
    \rho_{\sigma(1)} < \, &\rho_{\sigma(2)} + \rho_{\sigma(3)} + \rho_{\sigma(4)} < 1 + \rho_{\sigma(1)} \; \forall \sigma \in S_4, \\
    &\rho_{\sigma(1)} + \rho_4 < \rho_{\sigma(2)} + \rho_{\sigma(3)} \; \forall \sigma \in S_3, \label{eq:parWeightsTwoAndTwo} \\
    &\rho_1 + \rho_2 + \rho_3 + \rho_4 <1, \label{eq:SumParWeights}
\end{align}
which has solutions in an open subset of $(0, 1/2)^4$ that includes for example $\rho_1 = \rho_2 = \rho_3 = 1/4$,  $\rho_4 = 1/8$. As we have just seen, there is a $\CP ^1$ worth of parabolic bundles labelled by $w$.

In the case that $\mathcal{E} = \mathcal{O}(-1) \oplus \mathcal{O} (1)$, one can proceed in a similar vein. Because of~(\ref{eq:SumParWeights}) this bundle is destabilized by a degree $1$ holomorphic subbundle and will not appear in the moduli space. \\

\textbf{Case 2:} $\phi \neq 0$.

As before, $\mathcal{E} \simeq \mathcal{O}(-n) \oplus \mathcal{O}(n)$ and hence
\begin{equation}
    \End (\mathcal{E}) \simeq \begin{pmatrix} \O  & \O (-2n) \\ \O (2n) & \O \end{pmatrix}
\end{equation}
of which we aim to find a nontrivial section. Let us start by assuming that $\phi$ has non-vanishing residues at all 4 punctures. Since $\phi$ is furthermore nilpotent, we can trivialize the bundle away from a point, say $p_3 = \infty$, and write it as 
\begin{equation} \label{eq:nilpotentHiggs}
    \phi = \frac{dz}{z (z-1) (z-p)} \begin{pmatrix} a(z) \\ b(z) \end{pmatrix} \otimes \begin{pmatrix} -b(z) & a(z) \end{pmatrix} = \frac{dz}{z (z-1) (z-p)} \begin{pmatrix} -ab & a^2 \\ -b^2 & ab \end{pmatrix}
\end{equation}
for some polynomials $a,b \in \C [z]$. The condition that $\phi$ has a pole of first order at $z = \infty$ puts some constraints on the order of $a$ and $b$, which can be seen by switching to the $w = 1/z$ chart:
\begin{align} \begin{split}
        \phi &= \frac{-dw/w^2}{w(1-w)(1-wp)/w^4} \begin{pmatrix} -a(w^{-1})b(w^{-1}) & w^{2n}a(w^{-1})^2 \\ -w^{-2n}b(w^{-1})^2 & a(w^{-1})b(w^{-1}) \end{pmatrix} \\
    &= -\frac{-dw}{w(1-w)(1-wp)} \begin{pmatrix} -w^2 ab & w^{2n+2}a^2 \\ -w^{-2n+2} b^2 & w^2 ab \end{pmatrix}, \end{split}
\end{align}
from which we conclude $\deg a \leq 1+n$, $\deg b \leq 1-n$. This leaves the two cases $n=0$, i.e.\ $\mathcal{E} \simeq \O \oplus \O$, and $n=1$, i.e.\ $\mathcal{E} \simeq \O (-1) \oplus \O (1)$.

Let us start again with the case that $\mathcal{E} \simeq \O \oplus \O$. In this case both $a$ and $b$ are linear, say $a(z) = Az + B$ and $b(z) = Cz +D$. If the two were linearly dependent, say $a = \lambda b$ for some $\lambda$, then the line $\C \begin{pmatrix} 1, -\lambda \end{pmatrix}^T$ would be globally in the kernel, and the degree zero line bundle passing through it would destabilize $(\mathcal{E}, \phi)$. Hence, $a$ and $b$ are linearly independent, in which case the linear polynomial $a - \lambda b$ cannot have more than one root for any choice of $\lambda$. This forces the 4 flags to be mutually distinct, so that we can use our gauge freedom to bring them into the same form as in~(\ref{eq:4lines}). From there, it is easy to see that $a(z) = 1, b(z) = z, w=p$ and consequently
\begin{equation}
    \phi_p = \frac{dz}{z(z-1)(z-p)} \begin{pmatrix} z & -z^2 \\ 1 & -z \end{pmatrix}
\end{equation}
is the only solution (up to an overall scaling by $\zeta \in \mathbb{C}^\times$). The underlying vector bundle $\mathcal{E}$ with parabolic structure $w=p$ is called a \textit{wobbly bundle}: A stable vector bundle that admits a non-zero nilpotent Higgs field.

For the case $\mathcal{E} \simeq \mathcal{O}(-1) \oplus \O (1)$ we proceed similarly: $a$ is a quadratic polynomial while $b$ has to be constant. Assume that $b \neq 0$, then there is a gauge transformation such that
\begin{equation}
   \phi =  \frac{dz}{z (z-1) (z-p)} \begin{pmatrix} -b & a \\ 0 & 1/b \end{pmatrix} \begin{pmatrix} -ab & a^2 \\ -b^2 & ab \end{pmatrix} \begin{pmatrix} -1/b & a \\ 0 & b \end{pmatrix} = \frac{dz}{z (z-1) (z-p)} \begin{pmatrix} 0 & 0 \\ 1 & 0 \end{pmatrix}.
\end{equation}
This is the limit $\zeta . ( \O \oplus \O , \phi_p)$ as $\zeta \to \infty$ of the nilpotent Higgs bundle just discussed. Notice that the underlying parabolic vector bundle becomes unstable (as $\O (1)$ is a subbundle of positive parabolic degree). Indeed, this Higgs bundle is known as the \textit{uniformization point}: The point where the nilpotent cone and the Hitchin section intersect.

This is the full discussion if none of the residues vanish. Now let us assume that $ \phi$ is regular at some puncture, say at $p_3 = \infty$. In this case, the description~(\ref{eq:nilpotentHiggs}) holds but the subsequent discussion shows that $\phi$ has odd order at $\infty$. Another option is that $\mathrm{res} \phi = 0$ at two punctures, say $p_2 = 1$ and $p_3 = \infty$. Then $\phi$ takes the form
\begin{equation}
    \phi = \frac{dz}{z(z-p)} \begin{pmatrix} -ab & a^2 \\ -b^2 & ab \end{pmatrix}
\end{equation}
and regularity at $\infty$ implies that $\deg a \leq n$ and $\deg b \leq -n$, hence $n=0$. Thus the kernel is globally constant and there is a unique line bundle passing through it. Stability requires that the flags over the two punctures with vanishing residue of the Higgs field disagree with this line, in accordance with (\ref{eq:parWeightsTwoAndTwo}). In case that these two flags disagree with one another, we can bring the collection of flags into the general position of (\ref{eq:4lines}) with parabolic structure $w=0$, such that the Higgs bundle is
\begin{equation}
    \left( \mathcal{E}_0 = \mathcal{O} \oplus \mathcal{O} , \,   \phi_0 = \frac{ dz}{z(z-p)} \begin{pmatrix} 0 & 0 \\ 1 & 0 \end{pmatrix}, L_1 = \begin{pmatrix} 0 \\ 1 \end{pmatrix} = L_4, L_2 = \begin{pmatrix} 1 \\ 1 \end{pmatrix}, L_3 = \begin{pmatrix} 1 \\ 0 \end{pmatrix} \right).
\end{equation}
There is no gauge freedom left, so that there is a $\mathbb{C}^\times$ orbit of these given by rescaling $\phi_0 \mapsto \zeta \phi_0$. There is also the case that the flags over $1$ and $\infty$ agree, in which case we can choose a gauge where that flag is $(1 \; 0)^T$. There is now some gauge freedom left which can undo the $\mathbb{C}^\times$ action, so that there is in fact only a single point in the moduli space representing this equivalence class, and this point is given as the limit $\lim_{\zeta \to \infty} \zeta. (\mathcal{E}_0, \phi_0)$ of the bundle just discussed. Notice that the underlying parabolic vector bundle is unstable: The degree 0 line subbundle given by the first summand destabilizes it.

By analogy, there are two more $\mathbb{C}^\times$-families of Higgs bundles with nilpotent Higgs field, together with a point in the rescaling limit. Concretely, we have
\begin{align}
     &\left( \mathcal{E}_1 = \mathcal{O} \oplus \mathcal{O} , \,   \phi_1 = \frac{ dz}{(z-1)(z-p)} \begin{pmatrix} 1 & -1 \\ 1 & -1 \end{pmatrix}, L_1 = \begin{pmatrix} 0 \\ 1 \end{pmatrix}, L_2 = \begin{pmatrix} 1 \\ 1 \end{pmatrix} = L_4, L_3 = \begin{pmatrix} 1 \\ 0 \end{pmatrix} \right), \\
     &\left( \mathcal{E}_\infty = \mathcal{O} \oplus \mathcal{O} , \,   \phi_\infty = \frac{ dz}{(z-p)} \begin{pmatrix} 0 & 1 \\ 0 & 0 \end{pmatrix}, L_1 = \begin{pmatrix} 0 \\ 1 \end{pmatrix}, L_2 = \begin{pmatrix} 1 \\ 1 \end{pmatrix}, L_3 = \begin{pmatrix} 1 \\ 0 \end{pmatrix} = L_4 \right).
\end{align}
with parabolic structure $w=1$ resp.\ $w=\infty$.

To conclude our discussion, the nilpotent cone of the toy model has the following structure: There is a $\CP ^1$ worth of stable parabolic vector bundles parametrized by the parabolic structure $w$. For generic $w$, the bundle does not support a non-vanishing nilpotent Higgs field. However, for the special values $w \in D$ there is $\mathbb{C}^\times$-family of nilpotent Higgs fields supported on the corresponding parabolic bundle, parametrized by some scaling parameter $\zeta \in \mathbb{C}^\times$. Moreover, the large $\zeta$ limit exists and is another nilpotent Higgs bundle, where the underlying parabolic bundle is now unstable. Topologically, it hence has the structure of an affine $D_4$-configuration of $\CP ^1$'s (figure \ref{pic.toymodel}).

\begin{figure}[ht]
\begin{center}
\includegraphics[width=0.5\linewidth]{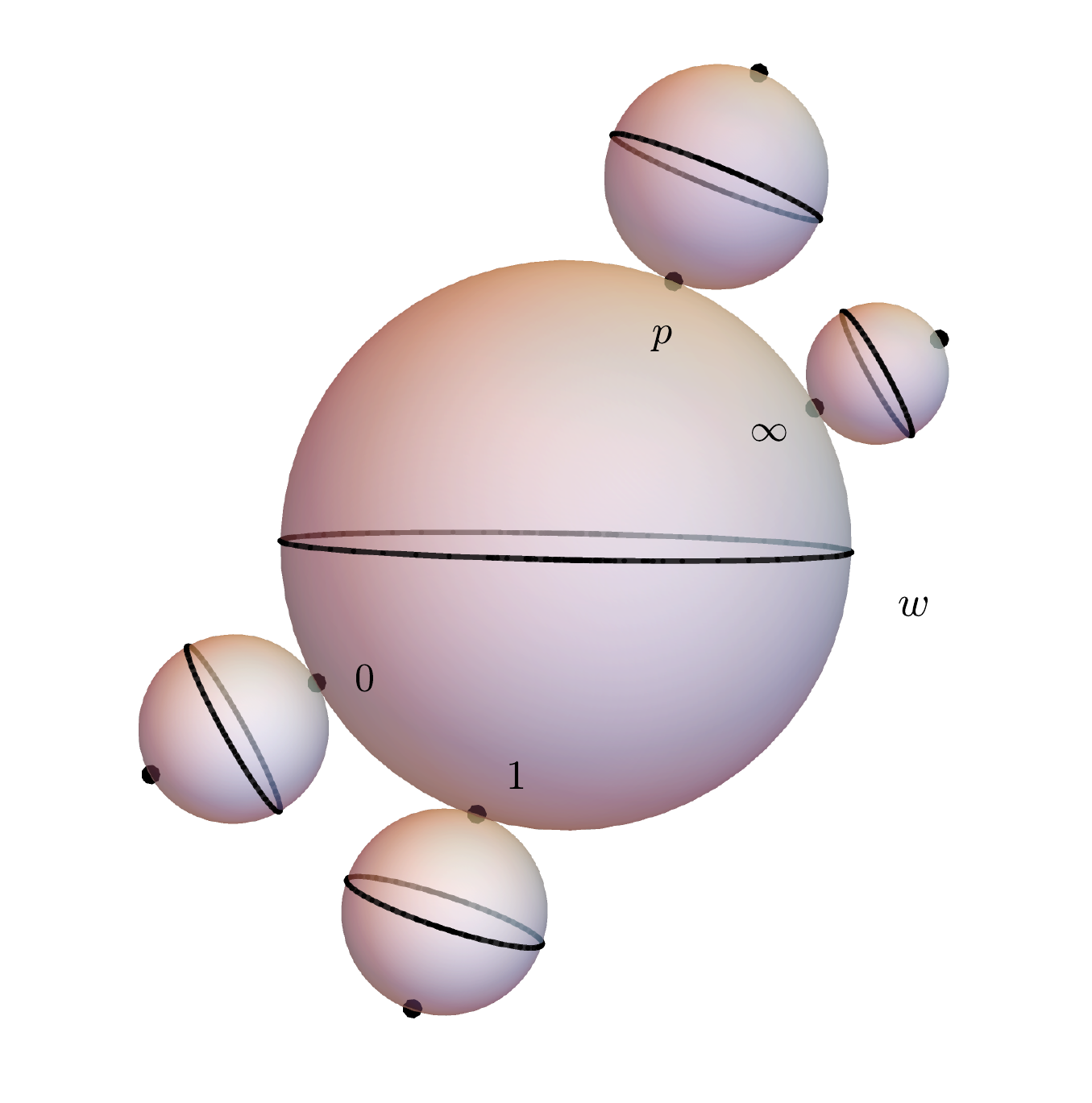}
%\captionsetup{labelformat=empty}
\caption{The nilpotent cone of the toy model.}\label{pic.toymodel}
\end{center}
\end{figure}

The locus of $\C$-VHS consists of the central $\CP^1$ (where the Higgs field is zero) as well as the four limit points at the tips of the external spheres. We know that to any other Higgs bundle inside of the nilpotent cone there is canonically associated a meromorphic quadratic differential, the space of which is one-dimensional and therefore uniquely determined up to a scalar $c \in \C$ by
\begin{equation}
    q_2 = \frac{c \cdot dz^2}{z (z-1) (z-p)}
\end{equation}
which indeed has a first-order pole at $z=\infty$. One could hope that there was an easy relationship between a $\C^\times$-orbit of a nilpotent Higgs field and this constant $c$ but making a quantitative statement would require determining the Chern connection (or at least a certain component of it) which means solving Hitchin's equations - a very difficult undertaking. The qualitative analysis from Proposition~\ref{prop:SmoothNotHolomorphic} still holds, meaning that the assignment is not holomorphic because the constant $c$ is bounded.

\begin{figure}[ht]
\begin{center}
\includegraphics[width=0.5\linewidth]{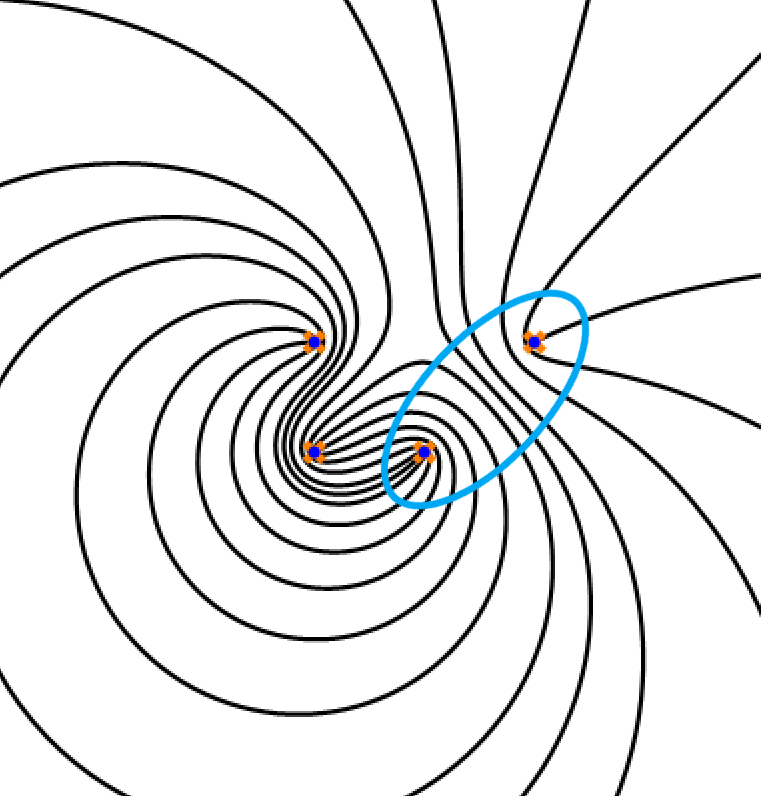}
%\captionsetup{labelformat=empty}
\captionsetup{width=.8\linewidth}
\caption{The trajectory structure for a generic quadratic differential in the toy model, and a WKB loop. The four first-order poles are at the marked locations $0$, $1$, $i$, $1+i$ and each emit a single ray. The genericity condition concerns the phase at which these rays are emitted.}\label{pic.SNtoymodel}
\end{center}
\end{figure}

%\newpage
\begin{appendices}
\section{WKB curves via translation surfaces} \label{app:WKBcurves}
In this appendix we will discuss the existence of WKB curves on Riemann surfaces (with holomorphic quadratic differentials). For this, we will remind the reader of the interpretation of such surfaces as \textit{half-translation surfaces} and construct a few explicit examples. Afterwards we will discuss the general theory. We do not claim any originality but were unable to find this discussed in the literature. %For more details on translation and half-translation surfaces, see e.g.~\cite{MassartSurvey}.

First, recall our definition of a WKB curve: Start with a connected Riemann surface $C$ with an $\mathrm{SL}(2,\C)$-Higgs bundle $(\delbar_E, \phi)$ such that $\phi$ is not nilpotent, then $\phi$ has eigenvalues $\pm \mu (z) dz$ that are generically nonzero. These define a branched double-covering $\Sigma \subset T^\ast C$ of $C$ known as the \textit{Spectral Curve} that keeps track of the two eigenvalues. Of course, $( \mu(z) dz )^2 = 1/2 \Tr (\phi^2)$, so the eigenvalues only depend on the quadratic differential $q_2$ associated to $(\delbar_E, \phi)$ via the Hitchin map, and then so does $\Sigma = \Sigma_{q_2}$. Around any point $z_0 \in C$ where $\mu (z_0) \neq 0,\infty$ one can define a local coordinate $w = w(z)$ via
\begin{equation}
    w(z) := \int_{z_0}^z \mu(z) dz = \int_{z_0}^z \sqrt{\frac{1}{2} \Tr(\phi(z)^2)} \hspace{0.2cm},
\end{equation}
and a change in base point corresponds to a constant translation of this local coordinate, while a different choice of square root (i.e.\ the exchange $\mu \leftrightarrow -\mu$) corresponds to an overall change of sign of $w$. 

These local coordinates give $C$ the structure of a \textit{half-translation surface}: Roughly speaking, a half-translation surface is a finite or countable collection of polygons in the Euclidean plane whose edges are identified by maps of the form $w \mapsto \pm w + a$ called \textit{half-translations}. Given such a half-translation surface, the inherited metric induces a complex structure on the glued surface, and the quadratic differential that is locally defined as $(dw)^2$ is invariant under half-translations descends to a globally well-defined quadratic differential.

There is one more subtlety that is worth pointing out for our discussion, namely the role of zeroes or poles of the (meromorphic) quadratic differential. In the description above, we allow zeroes of arbitrary order but poles only of order 1. While the neighborhood of a regular point is Euclidean as described above, zeroes or poles correspond to conical singularities on the side of half-translation surfaces: An opening angle larger than $(2+k) \pi $ corresponds to a zero of order $k$ ($k \geq 1$), while an opening angle of $\pi$ signifies a simple pole.

A special type of half-translation surface is a translation surface: All edge identifications preserve the orientation and are of the form $w \mapsto w + a$. These surfaces correspond to the case that the quadratic differential is globally the square of an Abelian differential $A \in H^0(C, K_C)$. We give some examples of translation and half-translation surfaces in Figures~\ref{fig:TranslationSurf} resp.~\ref{fig:HalfTranslationSurf}. Determining the genus of the corresponding surface is an easy exercise using the Euler characteristic: For example, the left surface in Figure~\ref{fig:TranslationSurf} contains 2 vertices, 12 edges (7 external and 5 internal) and 6 faces, yielding a surface of Euler characteristic $2-2g=2-12+6=-4$ and hence of genus $3$. Equivalently, one can count the order of zeroes and poles and determine this way the degree of the canonical bundle. Of course, the individual rectangles are not required to be of the same size (though sides that are to be identified need to be of the same length), and the complex structure of the resulting surface depends on these choices.

\begin{figure}[ht]
    \centering
   \includegraphics[width=0.9\linewidth]{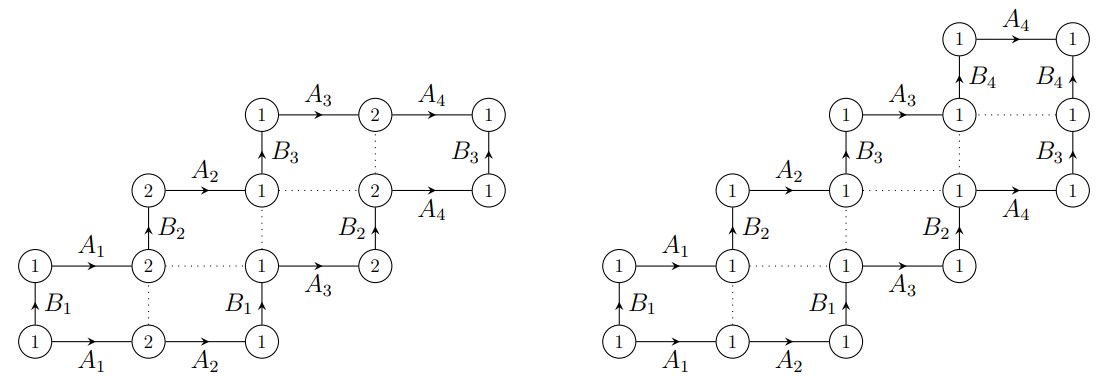}
   \captionsetup{width=.8\linewidth}
    \caption{Two translation surfaces. The labelling of the edges prescribes the gluing and induces an identification of the vertices which we have included for the convenience of the reader. A staircase composed of $2n$ rectangles as on the left and one with $2n-1$ as on the right both yield a surface of genus $n$.}
    \label{fig:TranslationSurf}
\end{figure}

Now we can come back to our original intention of constructing examples of WKB curves on Riemann surfaces $C$ equipped with a quadratic differential. Upon translating the problem to the language of half-translation surfaces, a curve $\gamma : [0,1] \rightarrow C$ becomes a collection of curves $\gamma_i : [t_i , t_{i+1}] \rightarrow P$ where $P$ is the set of polygons defining the half-translation structure, and $\gamma_i (t_i), \gamma_i (t_{i+1}) \in \partial P$ such that $\gamma_i (t_{i+1})$ and $\gamma_{i+1} (t_{i+1})$ are identified when gluing the edges of $P$ in the prescribed fashion. With this understood, a WKB curve is one for which $\Im \gamma_i$ is a strictly increasing function on $[t_i , t_{i+1}]$ for all $i$. It is now easy to see that the examples we have provided in figures~\ref{fig:TranslationSurf} and~\ref{fig:HalfTranslationSurf} have WKB curves given by vertical paths connecting two edges labeled by $A_i$.

\begin{figure}[ht]
    \centering
    \includegraphics[width=0.9\linewidth]{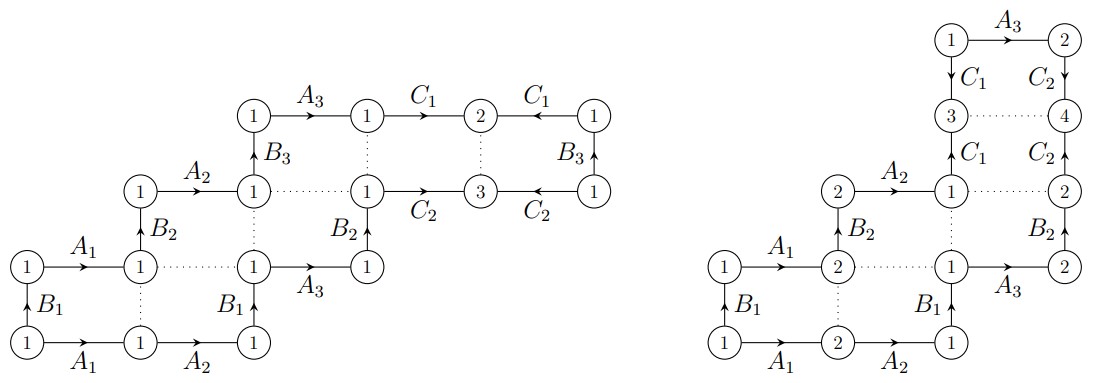}
    \captionsetup{width=.8\linewidth}
    \caption{Two half-translation surfaces: Notice that the vertices $2$ and $3$ in the left figure (resp. $3$ and $4$ in the right one) have opening angle $\pi$ and correspond to first-order poles of the quadratic differential. A configuration such as the one on the left with $2n+1$ rectangles yields a surface of genus $n$, the same is true for a configuration such as the one on the right consisting of $2n$ rectangles.}
    \label{fig:HalfTranslationSurf}
\end{figure}

Let $D \subset C$ be the set of zeroes and poles of $q_2$, and $C^o := C-D$. We are interested in curves $\gamma$ that evade the zeroes and poles of $q_2$ and so we are really interested in homotopy classes of loops $[\gamma] \in \pi_1 (C^o)$. In particular, this shows that two of examples considered in the previous paragraph are homotopic if and only if they connect the same two edges.

WKB loops have the special feature that their lift to the spectral curve $\Sigma$ consists of two disjoint loops. The reason is that $\Re (\mu(z)) > 0$ along a WKB curve, so the two sheets of the covering $\Sigma \to C$ cannot mix in a neighborhood of $\gamma$. It is therefore convenient to consider the notion of WKB curves on $\Sigma$ which is by itself a translation surface: The Liouville 1-form $\lambda$ on $T*C$ naturally restricts to an abelian differential on $\Sigma$.

Let us now discuss the general existence of WKB curves on half-translation surfaces where we will restrict our attention to quadratic differentials with at most simple poles. Notice that the foliation induced by $-q_2$ is equivalent to the one coming from $q_2$ (the segments are unoriented) but that for any $\theta \in \mathbb{R}/[0,\pi] \simeq \mathbb{RP}^1$ the horizontal foliation $\mathcal{F}_\theta$ coming from $e^{i \theta} q_2$ is transverse to the original $\mathcal{F} = \mathcal{F}_0$.

Now let $p \in C^o$, we want to construct a homologically non-trivial WKB loop that passes through $p$. It is a well-known fact (see e.g.\ \cite{Strebel}, \cite{GMN09}) that there are at most countably many angles $\theta$ for which the leaf of $\mathcal{F}_\theta$ passing through $p$ is non-recurrent so that generically the leaf forms a dense subset on $C$. It is then clear that in a small neighborhood of $p$, a leaf at a generic angle consists of parallel line segments and that one can smoothly connect a different line segment from the one that contains $p$ via a small path that remains transverse to $\mathcal{F}$ to form a loop. A sketch of this can be found in figure~\ref{pic.foliations}.

\begin{figure}[ht]
\begin{center}
\includegraphics[width=0.5\linewidth]{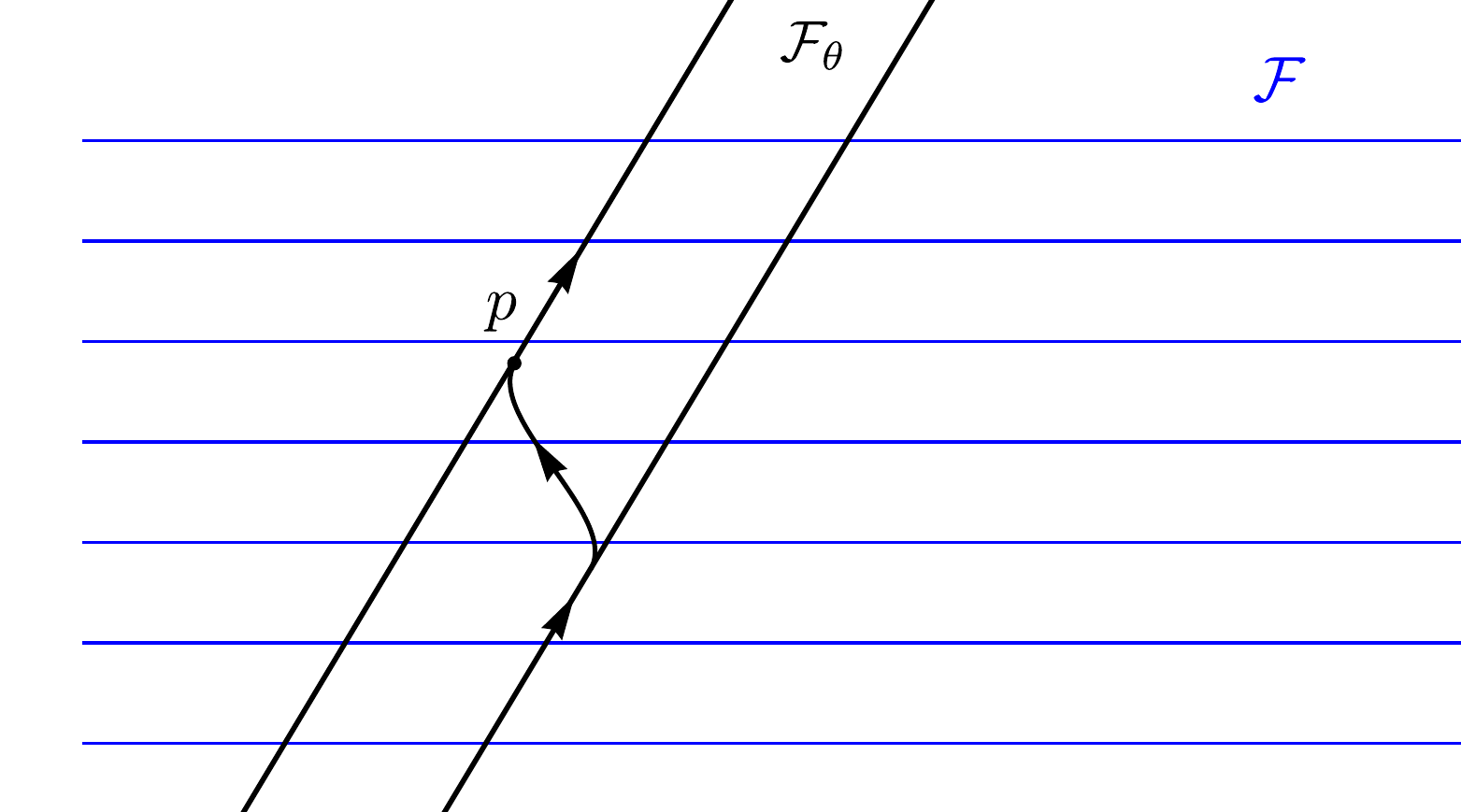}
%\captionsetup{labelformat=empty}
\captionsetup{width=.8\linewidth}
\caption{A minor detour connects two strands of a generic leaf of $\mathcal{F}_\theta$ to create a closed loop that is everywhere transverse to the foliation $\mathcal{F}$.}\label{pic.foliations}
\end{center}
\end{figure}

This carries a natural orientation when interpreted via its lift to $\Sigma$. Observe that $\gamma$ is necessarily homologically non-trivial as we shall explain now. If $\gamma$ were to close up without the detour, one could easily compute its period 
\begin{equation}
    Z_\gamma = l_\gamma \times e^{i\theta}
\end{equation}
where $l_\gamma$ is the length of $\gamma$. Since the period is nonzero and depends linearly on the homology class, this implies that it is homologically nontrivial. If we include the detour then the computation of the period is slightly more cumbersome but only differs marginally from the previous outcome, so in particular remains non-zero. This proves the following

\begin{prop} \label{prop:ExistenceWKBCurves}
Let $C$ be a compact Riemann surface and $q_2$ a meromorphic quadratic differential on $C$ that is non-zero and has at most simple poles. Then $(C, q_2)$ admits a WKB loop.
\end{prop}

Lastly, let us comment on the existence of WKB curves in the case of a non-nilpotent $\SL (n, \C)$-Higgs bundle $(\delbar_E, \phi)$. Let $u := h(\delbar_E, \phi)$ denote its projection to the Hitchin base and $\Sigma = \Sigma_u$ the associated spectral curve. Recall that $\Sigma_u \to C$ is a branched (n:1)-cover and that $(\Sigma_u , \lambda)$ is a translation surface. 

Pick a class $[\gamma ] \in \pi_1 (C)$ that admits a representative $\gamma: [0,1] \to C$ with the following properties:
\begin{enumerate}
    \item Its lift to $\Sigma_u$ is the disjoint union of $n$ loops;
    \item It has extension $\Tilde{\gamma}$ to a tubular neighborhood $U$ of $[0,1] \subset \C$ such that $\Tilde{\gamma}: U \to C$ is holomorphic. 
\end{enumerate}
$U$ is contractible and so $\Tilde{\gamma}^\ast \mathcal{E}$ is trivializable in a global coordinate $z$. The first condition implies that $\Tilde{\gamma}^\ast \phi$ is diagonalizable on $U$ with distinct eigenvalues $\mu_i (z) dz$. A WKB curve is one for which there is an ordering
\begin{equation}
    \Re (\mu_1) > \Re (\mu_2) > \dots > \Re ( \mu_n ) 
\end{equation}
for all $z \in U$. Every ordered pair $\{i,j | i > j \} \subset \{1, \dots, ,n \}$ determines an abelian differential $\lambda_{ij} = (\mu_i - \mu_j) dz$ on $U$ and $\gamma$ is a WKB curve if and only if it is a WKB curve with respect to each of the $\lambda_{ij}$. In other words, there are $\binom{n}{2}$ different foliations $\mathcal{F}_{ij}$ on $U$ and $\gamma$ has to be transverse to all of their leaves.

\begin{comment} We end by making two heuristic observations. The previous argument then generalizes as follows. Pick a point $p \in C$, then generically there is a neighborhood $V \subset C$ and a chart $V \hookrightarrow \C$ in which $p=0$ and $\mathcal{F}_i$ is a straight line of angle $\theta_i \neq \pm \pi/2$. For all but countably many angles $\theta$, the straight lin \end{comment}

Effectively, the asymptotics considered around~(\ref{eq:WKBHigherRank} should be governed by the largest (by real part) eigenvalue $\mu_1$, so even if some of the smaller eigenvalues were to exchange order along $\gamma$ it should not effect the asymptotic behavior. Consequently, the condition on $\gamma$ being eligible for the WKB method would require it to be transverse to $(n-1)$ foliations on $C$.

%Secondly, upon fixing a $p \in C$ that admits a neighborhood $V$ in which all of the foliations trivialize

\begin{comment}
In that case, the global nature of leaves of the associated foliation is well-known (see e.g.\ \cite{Strebel}, Chapter 11.4) and falls into exactly one of the following categories\footnote{The qualifying adjectives match those of~\cite{GMN09} but there is an unfortunate clash of language otherwise: What is called a \textit{WKB curve} there is a leaf of the foliation and thereby a curve transverse to what is called a WKB curve in this text.}:
\begin{enumerate}
    \item A \textit{generic} leaf is asymptotic in both directions to a pole.
    \item A \textit{finite} leaf is either closed or is asymptotic in both directions to a zero of $q_2$.
    \item A \textit{separating} leaf is asymptotic in one direction to a pole, and in the other direction to a zero.
    \item A \textit{divergent} leaf is not closed and does not have an asymptotic limit in at least one direction.
\end{enumerate}

Notice that the foliation induced by $-q_2$ is equivalent to the one coming from $q_2$ (the segments are unoriented) but that for any $\theta \in \mathbb{R}/[0,\pi] \simeq \mathbb{RP}^1$ the horizontal foliation $\mathcal{F}_\theta$ coming from $e^{i \theta} q_2$ is transverse to the original $\mathcal{F} = \mathcal{F}_0$. Now let $p \in C$ be neither a zero nor a pole of $q_2$, we want to construct a homologically non-trivial WKB loop that passes through $p$.
\end{comment}
\end{appendices}

%%%%%%%%%%%%%%%%%%%%%%%%%%%%%%%%%%%%%%%%%%%%%%%%

\bibliography{Main}{}

\newcommand{\etalchar}[1]{$^{#1}$}
\begin{thebibliography}{MSWW19}

\bibitem[AB20]{AllegrettiBridgeland}
D.~G.~L. Allegretti and T.~Bridgeland, ``The monodromy of meromorphic
  projective structures,'' \href{http://dx.doi.org/10.1090/tran/8093}{{\em
  Trans. Amer. Math. Soc.} {\bfseries 373} no.~9, (2020) 6321--6367}.

\bibitem[All19]{AllegrettiVorosSymbols}
D.~G.~L. Allegretti, ``Voros symbols as cluster coordinates,''
  \href{http://dx.doi.org/10.1112/topo.12106}{{\em J. Topol.} {\bfseries 12}
  no.~4, (2019) 1031--1068}.

\bibitem[BB73]{BBtheory}
A.~Bia{\l}ynicki-Birula, ``Some theorems on actions of algebraic groups,''
  \href{http://dx.doi.org/10.2307/1970915}{{\em Ann. of Math. (2)} {\bfseries
  98} (1973) 480--497}.

\bibitem[BDHH21]{BDHH21}
I.~Biswas, S.~Dumitrescu, L.~Heller, and S.~Heller, ``Holomorphic
  $\mathfrak{sl}(2,\mathbb{C})$-systems with {F}uchsian monodromy (with an
  appendix by {T}akuro {M}ochizuki),'' {\em arXiv preprint arXiv:2104.04818}
  (2021) .

\bibitem[BGP97]{BiquardGP}
O.~Biquard and O.~Garc\'{\i}a-Prada, ``Parabolic vortex equations and
  instantons of infinite energy,'' {\em J. Geom. Phys.} {\bfseries 21} no.~3,
  (1997) 238--254.

\bibitem[BHR19]{BHR}
I.~Biswas, S.~Heller, and M.~R\"{o}ser, ``Real holomorphic sections of the
  {D}eligne-{H}itchin twistor space,''
  \href{http://dx.doi.org/10.1007/s00220-019-03340-8}{{\em Comm. Math. Phys.}
  {\bfseries 366} no.~3, (2019) 1099--1133}.

\bibitem[Bri19]{Bridgeland_RH_from_DT}
T.~Bridgeland, ``Riemann-{H}ilbert problems from {D}onaldson-{T}homas theory,''
  \href{http://dx.doi.org/10.1007/s00222-018-0843-8}{{\em Invent. Math.}
  {\bfseries 216} no.~1, (2019) 69--124}.

\bibitem[BS15]{BridgelandSmith}
T.~Bridgeland and I.~Smith, ``Quadratic differentials as stability
  conditions,'' \href{http://dx.doi.org/10.1007/s10240-014-0066-5}{{\em Publ.
  Math. Inst. Hautes \'{E}tudes Sci.} {\bfseries 121} (2015) 155--278}.

\bibitem[Col16]{CollierThesis}
B.~Collier, {\em Finite order automorphisms of {H}iggs {B}undles: {T}heory and
  application}.
\newblock ProQuest LLC, Ann Arbor, MI, 2016.
\newblock Thesis (Ph.D.)--University of Illinois at Urbana-Champaign.

\bibitem[Cor88]{CorletteNAHT}
K.~Corlette, ``Flat {$G$}-bundles with canonical metrics,'' {\em J.
  Differential Geom.} {\bfseries 28} no.~3, (1988) 361--382.

\bibitem[CW19]{collier2019conformal}
B.~Collier and R.~Wentworth, ``Conformal limits and the {B}ia\l ynicki-{B}irula
  stratification of the space of {$\lambda$}-connections,''
  \href{http://dx.doi.org/10.1016/j.aim.2019.04.034}{{\em Adv. Math.}
  {\bfseries 350} (2019) 1193--1225}.

\bibitem[Del]{DeligneLetter}
P.~Deligne, ``Letters to {C}.\ {S}impson,''.

\bibitem[DFK{\etalchar{+}}21]{DFKMMN}
O.~Dumitrescu, L.~Fredrickson, G.~Kydonakis, R.~Mazzeo, M.~Mulase, and
  A.~Neitzke, ``From the {H}itchin section to opers through nonabelian
  {H}odge,'' \href{http://dx.doi.org/10.4310/jdg/1612975016}{{\em J.
  Differential Geom.} {\bfseries 117} no.~2, (2021) 223--253}.

\bibitem[DN19]{DumasNeitzke}
D.~Dumas and A.~Neitzke, ``Asymptotics of {H}itchin's metric on the {H}itchin
  section,'' \href{http://dx.doi.org/10.1007/s00220-018-3216-7}{{\em Comm.
  Math. Phys.} {\bfseries 367} no.~1, (2019) 127--150}.

\bibitem[Don87]{DonaldsonNAHT}
S.~K. Donaldson, ``Twisted harmonic maps and the self-duality equations,''
  \href{http://dx.doi.org/10.1112/plms/s3-55.1.127}{{\em Proc. London Math.
  Soc. (3)} {\bfseries 55} no.~1, (1987) 127--131}.

\bibitem[Dri81]{DrinfeldLetter}
V.~G. Drinfeld, ``Letter to {P}.\ {D}eligne,''(1981) .

\bibitem[FG06]{FockGoncharov}
V.~Fock and A.~Goncharov, ``Moduli spaces of local systems and higher
  {T}eichm\"{u}ller theory,''
  \href{http://dx.doi.org/10.1007/s10240-006-0039-4}{{\em Publ. Math. Inst.
  Hautes \'{E}tudes Sci.} no.~103, (2006) 1--211}.

\bibitem[FMSW20]{FMSW}
L.~Fredrickson, R.~Mazzeo, J.~Swoboda, and H.~Wei\ss, ``Asymptotic geometry of
  the moduli space of parabolic {$\rm{SL}(2, \C)$}-{H}iggs bundles,''(2020) .

\bibitem[Fre20]{Fredrickson18A}
L.~Fredrickson, ``Exponential decay for the asymptotic geometry of the
  {H}itchin metric,'' \href{http://dx.doi.org/10.1007/s00220-019-03547-9}{{\em
  Comm. Math. Phys.} {\bfseries 375} no.~2, (2020) 1393--1426}.

\bibitem[FZ02]{FZCluster}
S.~Fomin and A.~Zelevinsky, ``Cluster algebras. {I}. {F}oundations,''
  \href{http://dx.doi.org/10.1090/S0894-0347-01-00385-X}{{\em J. Amer. Math.
  Soc.} {\bfseries 15} no.~2, (2002) 497--529}.

\bibitem[Gai14]{GaiottoOpersAndTBA}
D.~Gaiotto, ``Opers and {TBA},'' {\em arXiv preprint arXiv:1403.6137} (2014) .

\bibitem[GMN13]{GMN09}
D.~Gaiotto, G.~W. Moore, and A.~Neitzke, ``Wall-crossing, {H}itchin systems,
  and the {WKB} approximation,'' {\em Advances in Mathematics} {\bfseries 234}
  (2013) 239--403.

\bibitem[Got94]{Gothen}
P.~B. Gothen, ``The {B}etti numbers of the moduli space of stable rank {$3$}
  {H}iggs bundles on a {R}iemann surface,''
  \href{http://dx.doi.org/10.1142/S0129167X94000449}{{\em Internat. J. Math.}
  {\bfseries 5} no.~6, (1994) 861--875}.

\bibitem[HH21]{HauselHitchin}
T.~Hausel and N.~Hitchin, ``Very stable higgs bundles, equivariant multiplicity
  and mirror symmetry,'' {\em arXiv preprint arXiv:2101.08583} (2021) .

\bibitem[Hit87a]{N1}
N.~J. Hitchin, ``The self-duality equations on a {R}iemann surface,''
  \href{http://dx.doi.org/10.1112/plms/s3-55.1.59}{{\em Proc. London Math. Soc.
  (3)} {\bfseries 55} no.~1, (1987) 59--126}.

\bibitem[Hit87b]{N2}
N.~J. Hitchin, ``Stable bundles and integrable systems,''
  \href{http://dx.doi.org/10.1215/S0012-7094-87-05408-1}{{\em Duke Math. J.}
  {\bfseries 54} no.~1, (1987) 91--114}.

\bibitem[Hit90]{HitchinWrongSign}
N.~J. Hitchin, ``Harmonic maps from a {$2$}-torus to the {$3$}-sphere,'' {\em
  J. Differential Geom.} {\bfseries 31} no.~3, (1990) 627--710.

\bibitem[IN14]{IwakiNakanishi}
K.~Iwaki and T.~Nakanishi, ``Exact {WKB} analysis and cluster algebras,''
  \href{http://dx.doi.org/10.1088/1751-8113/47/47/474009}{{\em J. Phys. A}
  {\bfseries 47} no.~47, (2014) 474009, 98}.

\bibitem[IN16]{IwakiNakanishi2}
K.~Iwaki and T.~Nakanishi, ``Exact {WKB} analysis and cluster algebras {II}:
  {S}imple poles, orbifold points, and generalized cluster algebras,''
  \href{http://dx.doi.org/10.1093/imrn/rnv270}{{\em Int. Math. Res. Not. IMRN}
  no.~14, (2016) 4375--4417}.

\bibitem[Kel10]{KellerClusterIntro}
B.~Keller, ``Cluster algebras, quiver representations and triangulated
  categories,'' in {\em Triangulated categories}, vol.~375 of {\em London Math.
  Soc. Lecture Note Ser.}, pp.~76--160.
\newblock Cambridge Univ. Press, Cambridge, 2010.

\bibitem[Lau88]{LaumonNilpotent}
G.~Laumon, ``Un analogue global du c\^{o}ne nilpotent,''
  \href{http://dx.doi.org/10.1215/S0012-7094-88-05729-8}{{\em Duke Math. J.}
  {\bfseries 57} no.~2, (1988) 647--671}.

\bibitem[MSWW16]{MSSW16}
R.~Mazzeo, J.~Swoboda, H.~Weiss, and F.~Witt, ``Ends of the moduli space of
  {H}iggs bundles,'' \href{http://dx.doi.org/10.1215/00127094-3476914}{{\em
  Duke Math. J.} {\bfseries 165} no.~12, (2016) 2227--2271}.
  \url{https://doi.org/10.1215/00127094-3476914}.

\bibitem[MSWW19]{MSSW}
R.~Mazzeo, J.~Swoboda, H.~Weiss, and F.~Witt, ``Asymptotic geometry of the
  {H}itchin metric,'' \href{http://dx.doi.org/10.1007/s00220-019-03358-y}{{\em
  Comm. Math. Phys.} {\bfseries 367} no.~1, (2019) 151--191}.

\bibitem[PPN19]{PPN}
C.~Pauly and A.~Pe\'{o}n-Nieto, ``Very stable bundles and properness of the
  {H}itchin map,'' \href{http://dx.doi.org/10.1007/s10711-018-0333-6}{{\em
  Geom. Dedicata} {\bfseries 198} (2019) 143--148}.

\bibitem[Sim88]{simpson88}
C.~T. Simpson, ``Constructing variations of {H}odge structure using
  {Y}ang-{M}ills theory and applications to uniformization,''
  \href{http://dx.doi.org/10.2307/1990994}{{\em J. Amer. Math. Soc.} {\bfseries
  1} no.~4, (1988) 867--918}.

\bibitem[Sim90]{SimpsonParabolicNHC}
C.~T. Simpson, ``Harmonic bundles on noncompact curves,''
  \href{http://dx.doi.org/10.2307/1990935}{{\em J. Amer. Math. Soc.} {\bfseries
  3} no.~3, (1990) 713--770}. \url{https://doi.org/10.2307/1990935}.

\bibitem[Sim92]{simpson92}
C.~T. Simpson, ``Higgs bundles and local systems,'' {\em Inst. Hautes
  \'{E}tudes Sci. Publ. Math.} {\bfseries 75} (1992) 5--95.

\bibitem[Sim97]{Sim97}
C.~T. Simpson, \href{http://dx.doi.org/10.1090/pspum/062.2/1492538}{``The
  {H}odge filtration on nonabelian cohomology,''} in {\em Algebraic
  geometry---{S}anta {C}ruz 1995}, vol.~62 of {\em Proc. Sympos. Pure Math.},
  pp.~217--281.
\newblock Amer. Math. Soc., Providence, RI, 1997.

\bibitem[Sim09]{Simpson09}
C.~T. Simpson, ``Katz's middle convolution algorithm,''
  \href{http://dx.doi.org/10.4310/PAMQ.2009.v5.n2.a8}{{\em Pure Appl. Math. Q.}
  {\bfseries 5} no.~2, Special Issue: In honor of Friedrich Hirzebruch. Part 1,
  (2009) 781--852}.

\bibitem[Sim10]{Simpson10}
C.~T. Simpson, \href{http://dx.doi.org/10.1090/conm/522/10300}{``Iterated
  destabilizing modifications for vector bundles with connection,''} in {\em
  Vector bundles and complex geometry}, vol.~522 of {\em Contemp. Math.},
  pp.~183--206.
\newblock Amer. Math. Soc., Providence, RI, 2010.

\bibitem[Str84]{Strebel}
K.~Strebel, {\em Quadratic Differentials}.
\newblock Springer Berlin Heidelberg, Berlin, Heidelberg, 1984.

\bibitem[Tho21]{ThomasRationalWKB}
A.~Thomas, ``Differential {O}perators on {S}urfaces and {R}ational {WKB}
  {M}ethod,'' {\em arXiv preprint arXiv:2111.07946} (2021) .

\bibitem[Vor83]{Voros83}
A.~Voros, ``The return of the quartic oscillator: the complex {WKB} method,''
  {\em Ann. Inst. H. Poincar\'{e} Sect. A (N.S.)} {\bfseries 39} no.~3, (1983)
  211--338.

\bibitem[Yok93]{Yoko}
K.~Yokogawa, ``Compactification of moduli of parabolic sheaves and moduli of
  parabolic {H}iggs sheaves,''
  \href{http://dx.doi.org/10.1215/kjm/1250519269}{{\em J. Math. Kyoto Univ.}
  {\bfseries 33} no.~2, (1993) 451--504}.

\end{thebibliography}
\bibliographystyle{fredrickson}

\smallbreak

\noindent {\bf  Sebastian Schulz - 
{\sc  University of Texas at Austin, USA.}\\
\tt  s.schulz@math.utexas.edu }

\end{document}